\documentclass[a4paper, 12pt]{article}

\usepackage{amssymb}
\usepackage{amstext}
\usepackage{amsthm}
\usepackage{amsmath}
\usepackage{psfrag}
\usepackage{graphicx}
\theoremstyle{definition}
\newtheorem{theorem}{Theorem}[section]
\newtheorem{corollary}[theorem]{Corollary}

\newtheorem{lemma}[theorem]{Lemma}
\newtheorem{definition}[theorem]{Definition}
\newtheorem{remark}[theorem]{Remark}
\newtheorem{remarks}[theorem]{Remarks}
\newtheorem{example}[theorem]{Example}

\DeclareMathOperator{\re}{Re}
\DeclareMathOperator{\im}{Im}

\title{An integral method for solving nonlinear eigenvalue problems}
\author{Wolf-J\"urgen Beyn \footnotemark[1] \\ 
Department of Mathematics, Bielefeld University \\
P.O. Box 100131, D-33501 Bielefeld}
\date{\today}
\begin{document}
\footnotetext[1]
{supported by CRC 701
'Spectral Structures and Topological Methods in Mathematics'.}

\maketitle

\begin{abstract}We propose a numerical method for computing all
eigenvalues (and the corresponding eigenvectors)
 of a nonlinear holomorphic eigenvalue problem that lie
within a given contour in the complex plane. The method uses complex
integrals of the resolvent operator, applied to at least $k$ column vectors,
where $k$ is the number of eigenvalues inside the contour. 
The theorem of Keldysh is employed to show that the original nonlinear 
eigenvalue problem reduces to a linear eigenvalue problem of dimension $k$.
 No initial approximations of eigenvalues and eigenvectors are needed.
The method is particularly suitable for moderately large eigenvalue problems
where $k$ is much smaller than the matrix dimension. We also give an extension
of the method to the case where $k$ is larger than the matrix dimension.
The quadrature errors caused by the trapezoid sum are discussed
for the case of analytic closed contours. Using
well known techniques it is shown that the error decays exponentially
with an exponent given by the  product of the number of quadrature
points and the minimal distance of the eigenvalues to the contour.
\end{abstract}
\section{Introduction}
We consider nonlinear eigenvalue problems of the form
\begin{equation} \label{e1.1}
T(z) v = 0, \quad v \in \mathbb{C}^m, v \neq 0, z \in \Omega,
\end{equation}
where $T:\Omega \rightarrow \mathbb{C}^{m,m}$ is assumed to be holomorphic in 
some  domain  $ \Omega \subset \mathbb{C}$.
The computation of all eigenvalues and eigenvectors inside $\Omega$ usually
requires to solve two problems (see \cite{mv04},\cite{bhms08} for recent
reviews) :
\begin{enumerate}
\item Approximate localization and separation of eigenvalues in suitable
domains resp. intervals,
\item accurate computation of eigenvalues and associated eigenvectors
by an iterative method.
\end{enumerate}
The global problem of localization can be substantially simplified if
minimum-maximum characterizations similar to the linear case hold
\cite{vw82},\cite{vo03}. Voss and co-workers have combined these principles
with locally convergent methods  
of Arnoldi or Jacobi-Davidson type (see \cite{v04},\cite{bv04},\cite{v07}),
and in this way provided an effective means for computing all eigenvalues.

Another case where both problems can be solved, 
is for polynomials 
\begin{equation*} \label{e1.3}
T(z) = \sum_{j=0}^{p} T_j (z-z_0)^j, 
\quad T_j \in \mathbb{C}^{m,m}.
\end{equation*}
This eigenvalue problem can be reduced to a linear eigenvalue
problem of dimension $pm$, and this is the path taken by the MATLAB
routine {\it polyeig}. Quite a few papers in the literature either
analyze this linearization approach or generalize methods from
linear eigenvalues to the polynomial case.

In the general holomorphic case we just have  a power series near each 
$z_0 \in \Omega$ 
\begin{equation*} \label{e1.2}
T(z) = \sum_{j=0}^{\infty} T_j (z-z_0)^j, \quad |z-z_0| \;\text{small},
\quad T_j \in \mathbb{C}^{m,m}.
\end{equation*}
One may then use polynomial truncation and the polynomial solver for
getting good initial estimates of the eigenvalues. However, the success
of this method strongly depends on the radius of convergence and
on the decay of the coefficient matrices. Also, it may be necessary
to compute power series at many different points in $\Omega$.

Finally, we refer to the recent approach of Kressner \cite{kr09}, who
uses the fact that any holomorphic matrix function can be written as
\begin{equation*} \label{krep}
T(z)= \sum_{j=1}^{p} f_j(z) T_j, \quad T_j \in \mathbb{C}^{m,m}
\end{equation*}
with holomorphic functions $f_j:\Omega \mapsto \mathbb{C}$ 
(such a representation always exists for some $p \le m^2$).
Then a Newton-type iteration is devised in \cite{kr09} that allows to
compute a group of eigenvalues and an associated subspace. Though
the convergence of this method is surprisingly robust to the
choice of initial values, it remains a method for solving
the local problem. 

In this paper we tackle the global problem by using contour integrals,
which seem to be the only available tool in the general holomorphic
case. The idea is to use the  theorem of Keldysh \cite{Ke51},\cite{Ke71}, 
which provides 
an expansion of $T(z)^{-1}$ in a neighborhood $\mathcal{U}\subset \Omega$ 
of an eigenvalue  $\lambda \in \Omega$ as follows:
\begin{equation} \label{e1.4}
T(z)^{-1} = \sum_{j=-\kappa}^{\infty} S_j (z-\lambda)^j, 
\quad z \in \mathcal{U}\setminus \{\lambda\}, \quad S_j \in \mathbb{C}^{m,m},
\quad S_{-\kappa}\neq 0.
\end{equation}
More specifically, Keldysh' theorem gives a representation of the singular
part in \eqref{e1.4} in terms of generalized eigenvectors of $T(z)$ and
its adjoint $T^H(z)$. A good reference for the underlying theory is
\cite{MM03} which we briefly review in Section \ref{sec2}. 

Numerical methods based on contour integrals seem not to have 
attracted much attention in the past. A notable exception 
are exponential integrators
and, more recently, approaches to compute analytic functions of matrices
via suitably transformed contour integrals, see
\cite{HHT08},\cite[13.3.2]{h08}. In  particular,the exponential 
convergence of the trapezoid sum is proved in \cite{HHT08}.

Our goal is to compute all eigenvalues and the associated eigenvectors
that lie within a given closed contour $\Gamma$  in $\Omega$. The main
algorithm is described in Section \ref{sec3}.
Suppose that  $k \le m$ eigenvalues of \eqref{e1.1} lie inside $\Gamma$.
Then our method reduces the nonlinear eigenvalue problem to a linear
one of dimension $k$ by evaluating the contour integrals
\begin{equation} \label{e1.5}
A_p = \frac{1}{2 \pi i} \int_{\Gamma} z^p T(z)^{-1} \hat{V} dz, \quad p=0,1.
\end{equation}
Here $\hat{V} \in \mathbb{C}^{m,k}$ is generally taken as a random matrix.
The contour integrals in \eqref{e1.5} are calculated approximately
 by the trapezoid sum. If $N$ quadrature points are used, this requires
to solve $Nk$ linear systems, which is the main numerical effort. 
As a consequence, our method is limited to moderately large nonlinear
eigenvalue problems for which a fast (sparse) direct solver is available.

In  Section \ref{sec4} we apply the algorithm to several examples,
showing that a moderate number of quadrature nodes ($N \approx 25$)
is usually sufficient to get good estimates of eigenvalues and eigenvectors.
Based on \cite{DR84}, we  prove in Section \ref{sec4} that the quadrature 
error decays  exponentially with an exponent that depends on the product of
the number of quadrature nodes and the smallest distance of the eigenvalues
to the contour. 

In the final Section \ref{sec5} we deal with two problems that are typical
for nonlinear eigenvalue problems and that do not occur in the linear case:
First, there can be much more eigenvalues than the
matrix dimension  (e.g. characteristic functions for 
delay equations) and, second, 
eigenvectors belonging to different eigenvalues can be linearly
dependent, even if the number of eigenvalues is less than the matrix
dimension.
In Section \ref{sec5} we extend our integral method such 
that it applies to the case $k >m$ and that it can also handle rank defects
of eigenspaces.
For the extended integral method it is necessary to evaluate $A_p$ from 
\eqref{e1.5}  for indices  $0 \le p \le 2\lceil \frac{k}{m} \rceil -1$.
Numerical examples show that this extension is suitable for solving both
aforementioned problems.
\smallskip

\noindent
{\bf Acknowledgement:} The author thanks Ingwar Petersen for the support
with the numerical experiments.

\section{Nonlinear eigenvalues and Keldysh' Theorem}
\label{sec2}
The material in this section is largely based on the monograph
\cite{MM03}. It contains a general study of meromorphic operator
functions that have values in spaces of Fredholm operators of index $0$.
For our purposes it is sufficient to consider  matrix valued mappings
\begin{equation*}\label{e2.1}	
		T: \Omega \subset \mathbb{C} \to \mathbb{C}^{m,m},
\end{equation*}
that are holomorphic in some open domain $\Omega$. 
We write this as $T\in H(\Omega,\mathbb{C}^{m,m})$.
For a matrix $A$ we denote by $R(A)$ and $N(A)$ its range and 
nullspace, respectively.
\begin{definition} \label{def1}
A number $\lambda \in \Omega$ is called an {\it eigenvalue} of $T(\cdot)$ if
$ T(\lambda)v =0 $ for some $v \in \mathbb{C}^m, v \neq 0$. The vector $v$
is then called a {\it (right) eigenvector}. By $\sigma(T)$ we denote the set of 
all eigenvalues and by $\rho(T)=\Omega \setminus \sigma(T)$ we denote
the {\it resolvent set}.\\
The eigenvalue $\lambda$ is called {\it simple} if
\[
N(T(\lambda))= \mathrm{span}\{v\}, v\neq 0
 \quad T'(\lambda)v \notin R(T(\lambda)).
\]
\end{definition}
\begin{theorem} \label{theinv}
Every eigenvalue $\lambda\in \sigma(T)$  of $T \in H(\Omega,\mathbb{C}^{m,m})$
is isolated, i.e. $\mathcal{U}\setminus\{\lambda\}\subset \rho(T)$ for some 
neigborhood $\mathcal{U}$ of $\lambda $.  \\
Moreover, $T(z)$ is meromorphic at $\lambda$, i.e. there exist
$\kappa \in \mathbb{N}$ and $S_j \in \mathbb{C}^{m,m}$ for $j \ge - \kappa$
 such that
$S_{-\kappa} \neq 0$ and 
\begin{equation} \label{e2.2}
T(z)^{-1} = \sum_{j=-\kappa}^{\infty} S_j (z-\lambda)^{j}, \quad
z \in \mathcal{U}\setminus \{\lambda\}.
\end{equation}
\end{theorem}
\begin{remark}
 The number $\kappa$ is uniquely determined and called the order
of the pole at $\lambda$.\\
The Theorem of Keldysh (see Theorem \ref{thekeldysh} below) gives a 
representation of the singular part
\begin{equation*}
\sum_{j=-\kappa}^{-1} S_j (z-\lambda)^j
\end{equation*}
in terms of (generalized) eigenvectors of $T$ and $T^H$. 
It goes back to
Keldysh \cite{Ke51} with a proof given in \cite{Ke71}.
Generalizations of Keldysh' theorem were derived by Trofimov \cite{Tr68},
who introduced the concept of root polynomials, and by Marcus and  Sigal
\cite{MS70} and Gohberg and Sigal \cite{GS71} who used factorizations of
operator functions. A simple direct proof was found by Mennicken and
M\"oller \cite{MM84} who later gave a concise approach to the whole
theory in \cite{MM03}.
\end{remark}

For the motivation of the algorithm in the next section it is instructive
to first state Keldysh' theorem for simple eigenvalues.
In this case Definition \ref{def1} implies for the  adjoint $T^H(z)$ 
\begin{equation*} \label{e2.3}
 N(T^H(\lambda))= \mathrm{span}\{w \} \quad \text{for some} \quad
w\in \mathbb{C}^{m}, w\neq 0,
\end{equation*}
\begin{equation*} \label{e2.4}
w^H T'(\lambda) v \neq 0.
\end{equation*}
Without loss of generality we can normalize $v$ and $w$ such that
\begin{equation} \label{e2.5}
w^H T'(\lambda) v =1.
\end{equation}
Then we are still free to further normalize either $|w|=1$ or
$|v|=1$. 

\begin{theorem} \label{Keldysh1}
	Assume $ \lambda \in \Omega $ is a simple eigenvalue of 
$T\in H(\Omega,\mathbb{C}^{m,m})$ with eigenvectors normalized as in 
(\ref{e2.5}).
Then there is a neighborhood $\mathcal{U} \subset \Omega$ of $\lambda$
and a holomorphic function $R \in H(\mathcal{U},\mathbb{C}^{m,m})$ such that
	\begin{equation} \label{expand1}
		T(z)^{-1} = \frac{1}{z-\lambda} v w^H+R(z), \quad 
z \in \mathcal{U}\setminus \{\lambda\}.
	\end{equation}
Moreover, let $\mathcal{C} \subset \Omega$ be a compact subset that contains
only simple eigenvalues $\lambda_n,n=1,\ldots,k$ with eigenvectors
 $v_n,w_n$ satisfying 
	\begin{equation} \label{normj}
	T(\lambda_n)v_n =0, \quad	w_n^HT(\lambda_n)=0, \quad 
 w_n^H T'(\lambda_n)v_n=1.
	\end{equation}
Then there is a neighborhood $ \mathcal{U}$ of $\mathcal{C}$ in $\Omega$ and a holomorphic
function $R \in H(\mathcal{U},\mathbb{C}^{m,m})$ such that
\begin{equation} \label{Tinv}
T(z)^{-1} = \sum_{n=1}^{k} \frac{1}{z - \lambda_n}v_n w_n^H + R(z), 
\quad z \in \mathcal{U}\setminus \{\lambda_1,\ldots,\lambda_k\}.
\end{equation}
\end{theorem}
\begin{proof}
The first part is a special case of Theorem \ref{thekeldysh} below.
For the second part, note that eigenvalues are isolated and hence we can
 choose a neighborhood
$\mathcal{C} \subset \mathcal{U} \subset \Omega$ such that 
$\sigma(T)\cap \mathcal{U} = \{ \lambda_1,\ldots,\lambda_k\}$.
Then the function
\begin{equation*}
R(z)= T(z)^{-1} - \sum_{n=1}^{k} \frac{1}{z - \lambda_n} v_n w_n^H
\end{equation*}
is holomorphic in $\mathcal{U} \cap \rho(T)$ and by the first part
it is also holomorphic in suitable neighborhoods of $\lambda_n,n=1,\ldots,k$.
\end{proof}

\begin{definition} \label{def2}
Let $T \in H(\Omega,\mathbb{C}^{m,m})$ and $\lambda \in \Omega$.
\begin{itemize}
\item[(i)] A function $v \in H(\Omega,\mathbb{C}^{m})$ is called a
{\it root function of $T$ at $\lambda$} if
\begin{equation*} \label{e2.7}
v(\lambda) \neq 0, \quad T(\lambda)v(\lambda) = 0.
\end{equation*}
The order of the zero $z=\lambda$ of $T(z)v(z)$ is called the { \it
multiplicity of $v$ at $\lambda$} and denoted by $s(v)$.
\item[(ii)] A tuple $(v_0,\ldots,v_{n-1})\in {\left(\mathbb{C}^m\right)}^n, n\ge 1$ is called
a {\it chain of generalized eigenvectors (CGE) of $T$ at $\lambda$}
if $v(z)=\sum_{j=0}^{n-1} (z-\lambda)^j v_j$ is a root function of $T$ at 
$\lambda$ of multiplicity $s(v) \ge n$.
\item[(iii)] For a given $v_0\in N(T(\lambda)), v_0 \neq 0$ the number
\begin{equation*}
r(v_0)= \max \{s(v): v \; \text{is a root function of} \;T\;
 \text{at}\; \lambda \;\text{with} \; v(\lambda)=v_0 \}
\end{equation*}
is finite and called the {\it rank of $v_0$}.
\item[(iv)] A system of vectors in $\mathbb{C}^m$
\begin{equation*}
V=\left(v_j^{\ell}, 0\le j \le m_{\ell}-1, 1 \le \ell \le L \right)
\end{equation*}
is called a {\it canonical system of generalized eigenvectors (CSGE) of $T$ at
$\lambda$}  if the following conditions hold:
\begin{enumerate}
\item[(a)] The vectors $v_0^1,\ldots,v_0^L$ form a basis of $N(T(\lambda))$,
\item[(b)] The tuple $(v_0^{\ell},\ldots,v_{m_{\ell}-1}^{\ell})$ is a CGE
of $T$ at $\lambda$ for $\ell=1,\ldots,L$,
\item[(c)] $m_{\ell}=\max \{ r(v_0): v_0\in N(T(\lambda))
\setminus \mathrm{span}\{ v_0^{\nu}:0\le \nu < \ell \} \}$
\\ for $\ell=1,\ldots,L$.
\end{enumerate}
\end{itemize}
\end{definition}

One can show that a CSGE always exists and that the numbers $m_{\ell}$
are ordered according to
\begin{equation*}
m_1 \ge m_2 \ge \ldots \ge m_L.
\end{equation*}
They are called the {\it partial multiplicities of $T$ at $\lambda$}.
With these notions we can state the following general theorem,
see \cite[Theorem 1.6.5]{MM03}.
\begin{theorem}[Keldysh] \label{thekeldysh}
Let $T\in H(\Omega,\mathbb{C}^{m,m})$ be given with $\rho(T) \neq \emptyset$.
For $\lambda \in \sigma(T)$ let
\begin{equation*}
V=\left(v_j^{\ell}, 0\le j \le m_{\ell}-1, 1 \le \ell \le L  \right)
\end{equation*}
be a CSGE of $T$ at $\lambda$.
Then there exists a CSGE
\begin{equation*}
W=\left(w_j^{\ell}, 0\le j \le m_{\ell}-1, 1 \le \ell \le L  \right)
\end{equation*}
of $T^H$ at $\lambda$, a neighborhood $\mathcal{U}$ of $\lambda$
and a function $R\in H(\mathcal{U},\mathbb{C}^{m,m})$ such that
\begin{equation} \label{Tinvfull}
T(z)^{-1} = \sum_{\ell=1}^{L}\sum_{j=1}^{m_{\ell}} (z- \lambda)^{-j}
\sum_{\nu=0}^{m_{\ell}-j} v_{\nu}^{\ell}w_{m_{\ell}-j-\nu}^{\ell H}
+ R(z), \quad z\in \mathcal{U}\setminus \{\lambda\}.
\end{equation}
The system $W$, for which \eqref{Tinvfull} holds, is the unique CSGE
of $T^H$ at $\lambda$ that satisfies the following conditions
\begin{equation*} \label{e2.8}
r(w_0^{\ell}) = m_{\ell}
\end{equation*}
\begin{equation} \label{e2.9}
\sum_{\alpha=0}^{j} \sum_{\beta=1}^{m_{\nu}} w_{j-\alpha}^{\ell H}
T_{\alpha+\beta}\, v_{m_{\nu}-\beta}^{\nu} = \delta_{\nu \ell}
\delta_{0 j},  0\le j \le m_{\ell}-1, 1\le \ell,\nu \le L,
\end{equation}
where
\begin{equation} 
T_j = \frac{1}{j !} T^{(j)}(\lambda), \quad j\ge 0.
\end{equation}
\end{theorem}
\begin{remark} Rather than using generalized eigenvectors one can
also write $T(z)^{-1}$ in terms of left and right root functions,
see \cite[Th.1.5.4]{MM03}.

The representation \eqref{Tinvfull} shows that the order $\kappa$ 
of the pole in \eqref{e2.2} is given by
\begin{equation*}
\kappa = \max\{m_{\ell}: \ell=1,\ldots,L \}.
\end{equation*}
Further, the number $L= \dim(N(T(\lambda)))$ is the geometric multiplicity
while $\sum_{\ell=1}^{L} m_{\ell}$ is the algebraic multiplicity
of $\lambda$.
In the semi-simple case $m_{\ell}=1,l=1,\ldots,L$, equations 
\eqref{Tinvfull} and \eqref{e2.9} simplify to
\begin{equation*} \label{Tinvsemi}
T(z)^{-1} = (z-\lambda)^{-1} \sum_{\ell=1}^{L} v_0^{\ell} w_{0}^{\ell H}
+ R(z),
\end{equation*}
 \begin{equation*} \label{orthsemi}
w_0^{\ell H} T'(\lambda) v_0^{\nu} = \delta_{\nu \ell}, \quad
1 \le \ell,\nu \le L,
\end{equation*}
which in  case $L=1$ further simplify to \eqref{expand1}
and \eqref{e2.5}.

\end{remark}

Consider now all eigenvalues inside a compact set  $\mathcal{C} \subset \Omega$.  
In the same way  as \eqref{Tinv} followed from  \eqref{expand1}, we obtain
from Theorem \ref{thekeldysh} the following corollary.
\begin{corollary} \label{cor1}
Let $\mathcal{C} \subset \Omega$ be compact and $T\in H(\Omega,\mathbb{C}^{m,m})$.
Then $\mathcal{C}$ contains at most finitely many eigenvalues $\lambda_n,n=1,\ldots,n(\mathcal{C})$
with corresponding CSGEs
\begin{equation*}
V_n=\left(v_j^{\ell,n}, 0\le j \le m_{\ell,n}-1, 1 \le \ell \le L_n  \right),
\quad n=1,\ldots,n(\mathcal{C}).
\end{equation*}
Let 
\begin{equation*}
W_n=\left(w_j^{\ell,n}, 0\le j \le m_{\ell,n}-1, 1 \le \ell \le L_n  \right),
\quad n=1,\ldots,n(\mathcal{C})
\end{equation*}
be the corresponding CSGEs of $T^H$ such that
\begin{equation*} \label{e2.10}
r(w_0^{\ell,n}) = m_{\ell,n}
\end{equation*}
and  with $T_{j,n}= \frac{1}{j!} T^{(j)}(\lambda_n)$
\begin{equation*} \label{e2.11}
\sum_{\alpha=0}^{j} \sum_{\beta=1}^{m_{\nu,n}} w_{j-\alpha}^{\ell,n H}
T_{\alpha+\beta,n}\, v_{m_{\nu,n}-\beta}^{\nu,n} = \delta_{\nu \ell}
\delta_{0 j}, 0\le j \le m_{\ell,n}-1, 1\le \ell,\nu \le L_n .
\end{equation*}
Then there exists a neighborhood $\mathcal{C} \subset \mathcal{U} \subset \Omega$
and a function $R\in H(\mathcal{U},\mathbb{C}^{m,m})$ such that
for all $z\in \mathcal{U}\setminus \{\lambda_1,\ldots,\lambda_{n(\mathcal{C})}\}$
\begin{equation*} \label{TinvK}
T(z)^{-1} = \sum_{n=1}^{n(\mathcal{C})} \sum_{\ell=1}^{L_n}\sum_{j=1}^{m_{\ell,n}} 
(z- \lambda)^{-j}
\sum_{\nu=0}^{m_{\ell,n}-j} v_{\nu}^{\ell,n}w_{m_{\ell,n}-j-\nu}^{\ell,n H}
+ R(z).
\end{equation*}
\end{corollary}
As a consequence of the corollary it follows that the order of the pole
in \eqref{e2.2} is given by
\begin{equation*} \label{e2.12}
\kappa = \max\{m_{\ell,n}: 0 \le \ell \le L_n, 1 \le n \le n(\mathcal{C}) \}.
\end{equation*}

Consider now a contour $\Gamma \subset \Omega$, i.e. a simple closed
 curve that has its interior $\mathrm{int}(\Gamma)$ in $\Omega$.
 An easy consequence of the residue theorem is the following result.

\begin{theorem} \label{intresolvent}
	Let $ T \in H(\Omega,\mathbb{C}^{m,m})$ have no eigenvalues on
the contour $\Gamma\subset \Omega$ and denote by  
$ \lambda_n,n=1,\ldots,n(\Gamma) $ the eigenvalues 
in the interior $\mathrm{int}(\Gamma)\subset \Omega$. 
Then with the CSGEs from Corollary \ref{cor1} we have for any 
$ f\in H(\Omega,\mathbb{C}) $  
\begin{equation} \label{residues}
	\frac{1}{2\pi i}\int_{\Gamma}f(z)T(z)^{-1}dz = 
\sum_{n=1}^{n(\Gamma)} \sum_{\ell=1}^{L_n} \sum_{j=1}^{m_{\ell,n}}
\frac{f^{(j-1)}(\lambda_n)}{(j-1)!} 
\sum_{\nu=0}^{m_{\ell,n}-j} v_{\nu}^{\ell,n}w_{m_{\ell,n}-\nu -j}^{\ell,n H}.
\end{equation}
If all eigenvalues are simple the formula reads
\begin{equation} \label{simpleres}
\frac{1}{2 \pi i} \int_{\Gamma} f(z) T(z)^{-1} dz =
\sum_{n=1}^{n(\Gamma)} f(\lambda_n) v_n w_n^H,
\end{equation}
where $v_n,w_n$ are left and right eigenvectors corresponding to $\lambda_n$
and  normalized according to
\begin{equation} \label{residuessimple}
w_{n}^{H} T'(\lambda_n) v_n =1, \quad n=1,\ldots,n(\Gamma).
\end{equation}
\end{theorem} 
\begin{proof} Corollary \ref{cor1} applies to 
$\mathcal{C}=\mathrm{int}(\Gamma) \cup \Gamma $, where the function 
$ f(z)T(z)^{-1}$  has residues  at $ \lambda_j $
given by the right-hand side of \eqref{residues}. The special case
$L_n=1,m_{0n}=1,n=1,\ldots,n(\Gamma)$ yields equation \eqref{simpleres}.
\end{proof}

\section{The algorithm for a few eigenvalues}
\label{sec3}
In the following we set up an algorithm for computing all eigenvalues
of $T\in H(\Omega,\mathbb{C}^{m,m})$ inside a given contour $\Gamma$
in $\Omega$. We assume that the sum of all algebraic multiplicities
\begin{equation} \label{kdef}
k = \sum_{n=1}^{n(\Gamma)}\sum_{\ell=1}^{L_n} m_{\ell,n}
\end{equation}
is less than or equal to the system dimension $m$. For the opposite case
we refer to Section \ref{sec5}. 
In high-dimensional problems we actually expect to have $k \ll m$.

\subsection{Simple eigenvalues inside the contour}
\label{sec3.1}

As in the second part of Theorem \ref{intresolvent}, let us assume
that all eigenvalues $\lambda_1,\ldots,\lambda_{n(\Gamma)}$ in 
$\mathrm{int}(\Gamma)$ are simple so that $k=n(\Gamma)$. We introduce
 the matrices
\begin{equation*} \label{vwmat}
V=\begin{pmatrix} v_1 \ldots v_{k} \end{pmatrix}, 
W=\begin{pmatrix} w_1 \ldots w_{k} \end{pmatrix} \in \mathbb{C}^{m,k}.
\end{equation*}
We assume that we have chosen a  matrix 
\begin{equation*} \label{wrand}
	\hat{V}\in \mathbb{C}^{m,l}, \quad k \le l \le m,
\end{equation*}
such that
\begin{equation} \label{Wcond}
W^H \hat{V} \in \mathbb{C}^{k,l} \quad \text{has rank} \quad k.
\end{equation}
In particular, this implies $\mathrm{rank}(W)=k$. In the applications we 
choose $\hat{V}$ at random (see Section \ref{sec4}), so that \eqref{Wcond} can be expected to
hold in a generic sense if $\mathrm{rank}(W)=k$. We note that
(in contrast to linear eigenvalue problems)
it is easy to construct nonlinear eigenvalue problems for which $W$ is rank
deficient. However, this seems to be a nongeneric situation for typical 
applications.
In addition to \eqref{Wcond} we assume
\begin{equation} \label{Vcond}
\mathrm{rank}(V) = k,
\end{equation}
which again is expected to hold in generic cases.

Next we compute the two integrals
\begin{equation} \label{A0}
	A_0 = \frac{1}{2\pi i}\int_{\Gamma}T(z)^{-1}\hat{V}dz \in \mathbb{C}^{m,l}
\end{equation}
\begin{equation} \label{A1}
	A_1 = \frac{1}{2\pi i}\int_{\Gamma}z T(z)^{-1}\hat{V}dz \in \mathbb{C}^{m,l}.
\end{equation}
The evaluation of these integrals by quadrature rules is by far the most
expensive part of the algorithm and will be discussed below.
Note also, that in the linear case $T(z)=zI-A$ the matrix $A_0$ is
obtained by applying to $\hat{V}$ the Riesz projector onto the invariant 
subspace associated with all eigenvalues inside $\Gamma$.
 
By \eqref{simpleres} we obtain
\begin{equation} \label{A0rep}
	A_0 = \sum_{n=1}^{k}v_nw_n^H\hat{V}=VW^H\hat{V}.
\end{equation}
Similarly,
\begin{equation} \label{A1rep}
	A_1 = \sum_{n=1}^{k}\lambda_nv_nw_n^T\hat{V}=V \Lambda W^H\hat{V}, \quad 
\Lambda=\text{diag}(\lambda_n,n=1,\ldots,k).
\end{equation}

In the next step we compute the singular value decomposition (SVD)
  of $A_0$ in reduced form
\begin{equation}
   VW^H\hat{V} = A_0= V_0\Sigma_0W_0^H
	\label{A0svd}
\end{equation}
where $
V_0\in \mathbb{C}^{m,k}, \Sigma_0=\text{diag}(\sigma_{1},\ldots,
\sigma_{k}),W_0\in\mathbb{C}^{l,k},V_0^HV_0=I_k,W_0^HW_0=I_k
$.
Note that the rank conditions \eqref{Wcond},\eqref{Vcond} show
that $\mathrm{rank}(A_0)=k$, hence $A_0$ has singular values
\begin{equation*} \label{singval}
 \sigma_{1} \geq \ldots \sigma_{k} > 0= \sigma_{k+1}=  \ldots 
= \sigma_{l}.
\end{equation*}
By the rank condition \eqref{Vcond} we have
\begin{equation*}
R(A_0)=R(V)=R(V_0).
\end{equation*}
Since both, $V_0$ and $V$ are $m \times k$ matrices and $V_0$ has orthonormal
columns, we obtain  
\begin{equation} \label{Srel}
V = V_0 S, \quad  S=V_0^H V \in \mathbb{C}^{k,k} \; \text{nonsingular}. 
\end{equation}
 With (\ref{A0rep}), (\ref{Srel}) we find
	$V_0SW^H\hat{V} = V_0\Sigma_0W_0^H$ and thus 
\begin{equation*}
W^H\hat{V} = S^{-1} \Sigma_0W_0^H.
\end{equation*} 
This relation is used to eliminate $W^H \hat{V}$ from
		$A_1= V_0 S \Lambda W^H\hat{V} $.
We obtain
\begin{equation*}
  V_0^H A_1= S \Lambda W^H \hat{V} =
		S \Lambda S^{-1}\Sigma_0W_0^H,
\end{equation*}
which upon multiplication by $ W_0 \Sigma_0^{-1}$ from the right 
 finally gives
\begin{equation} \label{A1end}
			S \Lambda S^{-1} = V_0^H A_1 W_0\Sigma_0^{-1}.
\end{equation}
Note that the right-hand side is a computable matrix which is diagonalizable
and has as eigenvalues exactly the eigenvalues of $T$ inside the contour.
We summarize the result in a theorem.
		
\begin{theorem} \label{alg1}
Suppose that $T \in H(\Omega,\mathbb{C}^{m,m})$ has only simple eigenvalues
$\lambda_1,\ldots,\lambda_k$ inside the contour $\Gamma$ in $\Omega$ 
with left and right eigenvectors
normalized as in \eqref{residuessimple}. Moreover, let a matrix 
$\hat{V} \in \mathbb{C}^{m,l}$ be given such that $k \le l \le m$ and
 the rank conditions \eqref{Wcond},\eqref{Vcond} are satisfied.
Then the matrix
\begin{equation} \label{beq}
B=V_0^H A_1 W_0\Sigma_0^{-1} \in \mathbb{C}^{k.k},
\end{equation}
given by \eqref{A0},\eqref{A1} and the SVD \eqref{A0svd},
is diagonalizable with eigenvalues $\lambda_1,\ldots,\lambda_k$.
From the eigenvectors $s_1,\ldots,s_k \in \mathbb{C}^{k}$ of $B$
one obtains the eigenvectors of $T$ through
\begin{equation*} \label{evrep}
v_n = V_0 s_n, \quad n=1,\ldots,k.
\end{equation*}
\end{theorem}
\begin{remarks} \label{rem3.2}
\noindent
(a) For reasons of numerical stability we may replace
$A_1$ by
\begin{equation*}
\tilde{A}_1 = \frac{1}{2 \pi i} \int_{\Gamma}(z-z_0)T(z)^{-1}\hat{V} dz
= A_1 - z_0 A_0.
\end{equation*}
For example, in case of a circle $\Gamma$, one can take $z_0$ as its
center.
Then \eqref{A1rep} holds with $\Lambda - z_0$ instead of $\Lambda$ and
the matrix $\tilde{B}= V_0^H \tilde{A}_1 W_0 \Sigma_0^{-1}$ has
eigenvalues $\lambda_n-z_0$. Therefore, the eigenvalues of $T$ are found by
adding $z_0$ to the eigenvalues of $\tilde{B}$.

\noindent
(b) The rank conditons in the theorem are crucial. Assume, for example,
that $A_0=V W^H \hat{V}$ has rank $k_0 < k$. Then the SVD \eqref{A0svd}
holds with matrices $W_0\in \mathbb{C}^{l,k_0}, V_0 \in \mathbb{C}^{m,k_0}$
and $\Sigma_0= \mathrm{diag}(\sigma_1,\ldots,\sigma_{k_0} )$. Moreover,
we have $S\in \mathbb{C}^{k_0,k}$ in \eqref{Srel}, 
$B \in \mathbb{C}^{k_0,k_0}$
in \eqref{beq}. Finally we find $B= S \Lambda \tilde{S}$ where 
$\tilde{S}= W^H\hat{V}W_0 \Sigma_0^{-1}$ satisfies $S \tilde{S}= I_{k_0}$.
Except for the case, when $S$ has some zero columns this does
not lead to a useful relation between the eigenvalues of $B$ and $\Lambda$.
For numerical computations we therefore recommend to test the residuals
$||T(\lambda_n)v_n||$, see Section \ref{sec3.3}. A general cure of
this rank deficient case is provided by the generalized algorithm
in Section \ref{sec5} which, however, is computationally more
expensive.
\end{remarks}
\subsection{Multiple eigenvalues inside the contour}
\label{sec3.2}
Let us consider the general case where $T \in H(\Omega,\mathbb{C}^{m,m})$
has no eigenvalues on the contour $\Gamma$ but may have multiple eigenvalues
inside. We apply Corollary \ref{cor1}
to the compact set $\mathcal{C}=\Gamma \cup \mathrm{int}(\Gamma)$ and assume that 
the matrix composed of all CSGEs that belong to eigenvalues inside $\Gamma$,
\begin{equation} \label{vdef}
V=\left(v_j^{\ell,n}, 0\le j \le m_{\ell,n}-1, 1 \le \ell \le L_n,
1 \le n \le n(\Gamma)   \right),
\end {equation}
has rank $k$, cf. \eqref{kdef}. Then, using  Theorem \ref{intresolvent}
with $f(z)=1$ shows that $A_0$, as defined in \eqref{A0}, satisfies
\begin{equation*} \label{A0repmult}
A_0 = \sum_{n=1}^{n(\Gamma)} \sum_{\ell=1}^{L_n} \sum_{\nu=0}^{m_{\ell,n}-1} 
v_{\nu}^{\ell,n}w_{m_{\ell,n}-1-\nu}^{\ell,n H} \hat{V}.
\end{equation*} 
Further, we assume that the matrix
\begin{equation} \label{nondegmult}
W^H \hat{V} \in \mathbb{C}^{k,l}
\end{equation}
has maximum rank $k$, where 
\begin{equation} \label{wdef}
W=\left(w_{m_{\ell,n}-1-\nu}^{\ell,n}, 0\le \nu \le m_{\ell,n}-1, 1 \le \ell \le L_n,
1 \le n \le n(\Gamma)   \right) \in \mathbb{C}^{m,k},
\end{equation}
is normalized as in Theorem \ref{thekeldysh}.
With Theorem \ref{intresolvent} we then find
\begin{equation*} \label{A1repmult}
A_1 = \sum_{n=1}^{n(\Gamma)} \sum_{\ell=1}^{L_n}\left[ 
\lambda_n 
\sum_{\nu=0}^{m_{\ell,n}-1}  v_{\nu}^{\ell,n}w_{m_{\ell,n}-1-\nu}^{\ell,n H}
+\sum_{\nu=0}^{m_{\ell,n}-2}  v_{\nu}^{\ell,n}w_{m_{\ell,n}-2-\nu}^{\ell,n H} 
 \right] \hat{V}
= V \Lambda W^H \hat{V},
\end{equation*} 
where $\Lambda$ has Jordan normal form
\begin{equation} \label{lambdanormal}
\Lambda = \begin{pmatrix} J_1 & & \\
                      & \ddots & \\
 & & J_{n(\Gamma)}  
           \end{pmatrix}, 
 \;
J_n = \begin{pmatrix} J_{n,1} & & \\
                      & \ddots & \\
 & & J_{n,L_n}  
           \end{pmatrix},
\;
J_{n,\ell} = \begin{pmatrix} \lambda_n &1 & \\
                      & \ddots & \ddots\\
 & & \lambda_n  
           \end{pmatrix}.
\end{equation}
As in Section \ref{sec3.1} the next steps are the SVD \eqref{A0svd}
for $A_0$ and the computation of 
$B=V_0^H A_1 W_0\Sigma_0^{-1} \in \mathbb{C}^{k.k}$.
Then  $B$ has eigenvalues $\lambda_1,\ldots,\lambda_{n(\Gamma)}$
and its Jordan normal form has the same partial multiplicities
as $T(z)$.

\begin{theorem} \label{algmult}
Suppose that $T \in H(\Omega,\mathbb{C}^{m,m})$ has no eigenvalues on
the contour $\Gamma$ in $\Omega$ and pairwise distinct eigenvalues
$\lambda_n,n=1,\ldots,n(\Gamma)$ inside $\Gamma$ with partial
multiplicities $m_{1,n} \ge \ldots \ge m_{L_n,n}, n=1,\ldots,n(\Gamma)$.
Moreover, assume that the matrix of generalized eigenvectors from \eqref{vdef}
and the matrix $W^H \hat{V}$ from \eqref{nondegmult} have rank $k$ with 
$k$ given by \eqref{kdef}.
 Then the matrix $B\in \mathbb{C}^{k,k}$ from \eqref{beq} has Jordan normal
form \eqref{lambdanormal} with the same eigenvalues $\lambda_n$
and partial multiplicities   $m_{\ell,n}$ 
($\ell=1,\ldots,L_n,n=1,\ldots,n(\Gamma)$). 
Suitable CSGEs for $T$ can be obtained from corresponding CSGEs 
$s_j^{\ell,n}$ for $B$ via
\begin{equation*} \label{gevrep}
v_j^{\ell,n} = V_0 s_j^{\ell,n}, \quad 0\le j \le m_{\ell,n}-1, 
1\le \ell \le L_n, 1\le n \le n(\Gamma).
\end{equation*}
\end{theorem}
\begin{remark} Essentially, the theorem reduces  the nonlinear 
problem  for eigenvalues inside a contour to a linear eigenvalue
problem for a $k \times k$-matrix. The linear eigenvalue problem inherits
the multiplicity structure of the nonlinear problem. As usual, computing
the Jordan normal form is not a stable process 
and other forms, such as the Schur form, are recommended.
A closer look at the derivation of the algorithm \eqref{A0svd},\eqref{A1end}
shows that it is sufficient to have a rank revealing $QR$-decomposition.
One would then replace $W_0 \Sigma_0^{-1}$ in \eqref{A1end} by the inverse
of the maximum rank upper triangular submatrix.
 \end{remark}
\subsection{Quadrature and numerical realization}
\label{sec3.3}
The major step in the algorithm consists in evaluating the integrals
\eqref{A0} and \eqref{A1} by numerical quadrature and by solving the linear
systems involved in the evaluation of the integrand.
We assume that $\Gamma$ has a $2 \pi$-periodic smooth parameterization
\begin{equation*} \label{param}
\varphi \in C^1(\mathbb{R},\mathbb{C}), \quad 
\varphi(t+2 \pi)=\varphi(t) \quad \forall t\in \mathbb{R}.
\end{equation*}
Of particular interest is the real analytic case $\varphi \in
C^{\omega}(\mathbb{R},\mathbb{C})$.
Taking equidistant nodes $ t_j=\frac{2 j \pi}{N}, j=0,\ldots,N$
and using the trapezoid sum, we find the following approximations
\begin{equation} \label{trap0}
\begin{aligned}
A_0 =&  \frac{1}{2 \pi i} \int_0^{2 \pi} T(\varphi(t))^{-1} \hat{V} 
\varphi'(t) dt  \approx \\
A_{0,N} = & \frac{1}{iN} \sum_{j=0}^{N-1} T(\varphi(t_j))^{-1} \hat{V}
 \varphi'(t_j),
\end{aligned}
\end{equation}
where we used $\varphi(t_0)=\varphi(t_N)$.
Similarly,
\begin{equation} \label{trap1}
A_1 \approx A_{1,N} = \frac{1}{iN} \sum_{j=0}^{N-1}T(\varphi(t_j))^{-1} \hat{V}
\varphi(t_j) \varphi'(t_j).
\end{equation}
In order to compute $A_{0,N}$ we need to solve $Nl$ linear systems
with $N$ different matrices $T(\varphi(t_j)),j=0,\ldots,N-1$
and with $l$ different right-hand sides each.
Note that we can use the solutions of these linear systems to compute
$A_{1,N}$ at almost no extra cost.
For the special case of a circle $\varphi(t)=
\mu + R e^{it}$ we obtain the formulas 
\begin{equation*} \label{quadcirc}
\begin{aligned}
A_{0,N} = &\frac{R}{N} \sum_{j=0}^{N-1} T(\varphi(t_j))^{-1} \hat{V}
\exp(\frac{ 2 \pi i j}{N}) , \\
A_{1,N} = & \mu A_{0,N} + \frac{R^2}{N} \sum_{j=0}^{N-1} T(\varphi(t_j))^{-1} \hat{V}
\exp(\frac{ 4 \pi i j}{N}).
\end{aligned}
\end{equation*}
The algorithm can be summarized as follows:
\smallskip

\noindent
{\bf Integral algorithm 1}
\smallskip

\noindent
{\bf Step 1:} Choose an index $l \le m$ and a matrix 
$\hat{V}\in \mathbb{C}^{m,l}$ at random.
\smallskip

\noindent 
{\bf Step 2:} Compute $A_{0,N}$,$A_{1,N}$ from \eqref{trap0},\eqref{trap1}.
\smallskip

\noindent
{\bf Step 3:} Compute the SVD $A_{0,N}=V \Sigma W^H$, where \\
$V\in \mathbb{C}^{m,l}$, $W \in \mathbb{C}^{l,l}$, $V^HV=W^HW=I_l$, 
$\Sigma=\mathrm{diag}(\sigma_1,\sigma_2,\ldots,\sigma_l)$.
\smallskip

\noindent
{\bf Step 4:} Perform a rank test for $\Sigma$, i.e.
 find $0<k\le l$ such that \\
$\sigma_1 \ge \ldots \ge \sigma_k > \mathrm{tol}_{\mathrm{rank}}
> \sigma_{k+1} \approx \ldots
\approx \sigma_l \approx 0 $. \\
If $k=l$ then increase $l$ and go to Step 1.\\
Else let $V_0=V(1:m,1:k),W_0=W(1:l,1:k)$ and \\ 
$\Sigma_0=\mathrm{diag}(\sigma_1,\sigma_2,\ldots,\sigma_k)$.
\smallskip

\noindent
{\bf Step 5:} Compute $B=V_0^H A_{1,N} W_0 \Sigma_0^{-1} \in 
\mathbb{C}^{k,k}$.
\smallskip

\noindent
{\bf Step 6:} Solve the eigenvalue problem for $B$ \\
$BS =S \Lambda$, $S=(s_1 \ldots s_k), \Lambda=
\mathrm{diag}(\lambda_1,\ldots,\lambda_k)$.\\
If $||T(\lambda_j)v_j||\le \mathrm{tol}_{\mathrm{res}}$  and 
$\lambda_j \in \mathrm{int}(\Gamma)$
accept $v_j=V_0 s_j$ as eigenvector and $\lambda_j$ as eigenvalue. 

\begin{remarks}
(a) If we find $k=l$ positive singular values in Step 4 then we take
this as an indication that there may be more than $l$ eigenvalues
(including multiplicities) inside $\Gamma$. We then increase $l$ until
a rank drop is detected in Step 4.\\
(b) In general, it is more efficient to compute 
$A_{1,N}$ in Step 5, when the index $k$ has 
been determined. Then one has to store the solutions of the
linear systems solved during the evaluation of $A_{0,N}$.
\\
(c) As noted in Remark \ref{rem3.2}(b) the algorithm may fail due
to linear dependency of (generalized) eigenvectors. Therefore, we
include a test of the residual. Moreover, as the experiments in 
Section \ref{sec4} show, eigenvalues close to the contour, either inside
or outside $\Gamma$,  may lead to difficulties in the rank test.
Therefore, the trivial test $\lambda_j \in \mathrm{int}(\Gamma)$ 
is included in Step 6 as well.
\\
(d) In Step 6 we assumed that eigenvalues are simple. If multiplicities
occur or $B$ is only brought into upper triangular form, then the eigenvalues
can still be read off from the diagonal, and the structure of eigenvectors
can be retrieved from $V_0 S$.
\end{remarks}
\section{Error analysis and numerical examples  }
\label{sec4}
\subsection{Error analysis}
\label{sec4.1}
Standard results on the trapezoid sum for holomorphic periodic integrands
imply exponential convergence at a rate that depends on the
number of nodes times the width of the horizontal strip of holomorphy,
see \cite{D59},\cite[4.6.5]{DR84}. Applications of these results to
the computation of matrix functions via contour integrals appear 
in \cite{HHT08}.
\begin{theorem} 
\label{thestrip}
Let $f \in H(S(d_-,d_+), \mathbb{C})$ be $2\pi$-periodic on the strip
\begin{equation*} \label{err1}
S(d_-,d_+) = \{ z \in \mathbb{C}: - d_- < \im z < d_+ \},
\quad d_{\pm} > 0.
\end{equation*}
Then the error of the trapezoid sum
\begin{equation*}
\label{err2}
E_N(f)= \frac{1}{2\pi}\int_0^{2 \pi} f(x) dx 
- \frac{1}{N} \sum_{j=0}^{N-1} f(\frac{2 \pi j}{N})
\end{equation*}
satisfies for all $0 < r_- < d_-, 0<r_+ < d_+$
\begin{equation*} \label{err3}
|E_N(f)| \le \max_{\im(z)=r_+}|f(z)| \; G(e^{-N r_+}) +
             \max_{\im(z)=r_-}|f(z)| \; G(e^{-N r_-}),
\end{equation*}
where $G(x)= \frac{x}{1-x}, x \neq 1$.
\end{theorem}
\begin{remark} Note that Theorem \ref{thestrip} is a slight variation
of \cite[4.6.5]{DR84} since $f$ is not assumed to be real on  $[0,2\pi]$
and the strip $S(d_-,d_+)$ can be unsymmetric, in general.
\end{remark}

In the following we state and prove the corresponding result for integrals over
circles which will be used in the sequel.

\begin{theorem} \label{thecircle}
Let $f \in  H(A(a_-,a_+),\mathbb{C})$ be holomorphic on the annulus
\begin{equation*} \label{err4}
A(a_-,a_+)= \{ z \in \mathbb{C} : 
\frac{1}{a_-} < \frac{|z|}{R} < a_+ \}, \quad a_{\pm} > 1,
\end{equation*}
for some $R > 0$. Then the error of the trapezoid sum 
\begin{equation}
\label{err5}
E_N(f)= \frac{1}{2\pi i}\int_{|z|=R} f(z) dz 
- \frac{R}{N} \sum_{j=0}^{N-1} f(R \omega_N^j)\omega_N^j, \quad
\omega_N = \exp(\frac{2 \pi i}{N}),
\end{equation}
satisfies for all $1 < \rho_- < a_-, 1 < \rho_+ < a_+$
\begin{equation} \label{err6}
|E_N(f)| \le \max_{|z|=\rho_+ R}|f(z)| \; G(\rho_+^{-N})) +
             \max_{\rho_- |z|=R}|f(z)| \; G(\rho_-^{-N}).
\end{equation}
\end{theorem}
\begin{proof} We use the Laurent expansion of $f$ (see e.g. \cite{GK06})
\begin{equation} \label{err7}
f(z) = \sum_{k=-\infty}^{\infty} f_k z^k, \quad
f_k = \frac{1}{2 \pi i} \int_{|z|=R} f(z) z^{-k-1} dz,
\end{equation}
which converges uniformly on compact subdomains of the annulus.
By a simple computation,
\begin{equation*} \label{err8}
E_N(z^k) = \left\{ \begin{array}{cc}
                  - R^{\ell N}, & k+1 = \ell N, 
\ell \in \mathbb{Z} \setminus\{0\}, \\
0 & \text{otherwise}.
\end{array}
\right.
\end{equation*}
Applying $E_N$ to \eqref{err7} leads to
\begin{equation} \label{err9}
E_N(f)= - \sum_{\ell=1}^{\infty} (f_{\ell N}R^{\ell N} + f_{-\ell N} R^{-\ell N}).
\end{equation}
From Cauchy's Theorem and a standard estimate we obtain
\begin{equation*}
\begin{array}{rl}
|f_{\ell N}R^{\ell N}| = & \left| \frac{R^{\ell N}}{2 \pi i} \int_{|z|= R} f(z) z^{-\ell N -1}
 dz \right| \\
 = & R^{\ell N} \left| \frac{1}{2 \pi i} \int_{|z|= \rho_+ R} f(z) z^{-\ell N -1}
 dz\right| \\
\le & \frac{R^{\ell N}}{2 \pi} 2 \pi \rho_+ R \max_{|z|=\rho_+ R}|f(z)|
\left( \rho_+ R \right)^{-\ell N -1}  \\
= & \max_{|z|=\rho_+ R}|f(z)| \; \rho_+^{- \ell N}.
\end{array}
\end{equation*}
In a similar way,
\begin{equation*}
|f_{-\ell N} R^{-\ell N}| \le \max_{\rho_- |z| = R }|f(z)| \; \rho_-^{- \ell N}.
\end{equation*}
Using these estimates in \eqref{err9} completes the proof.
\end{proof}
The proof shows that the $\rho_-$-term can be discarded in \eqref{err6}
if the principal term in the Laurent expansion vanishes (i.e. $f_k=0$ for
$k\le -1$). Likewise, the $\rho_+$-term disappears when $f_k=0$ for $k\ge 0$.
For the function 
\begin{equation} \label{err10}
f(z)= (z -\lambda)^{-j}, \quad j\ge 1,
\end{equation}
the principal term vanishes for $| \lambda| > R$ while the secondary
term vanishes for $ |\lambda| <R$. Example \eqref{err10} is crucial
for the application to the meromorphic functions from Section \ref{sec3}.
Therefore, we note the following explicit formula.

\begin{lemma} \label{propsing}
The error of the trapezoid sum \eqref{err5} for the function \eqref{err10}
in case $N\ge j$  is given as follows,
\begin{equation} \label{err11}
E_N((z-\lambda)^{-j}) = \frac{(-1)^{j-1}\lambda^{-j}}{(j-1)!}
\left\{ \begin{array}{lc}
\frac{d^{j-1}}{dx^{j-1}}(x^{j-1} G(x^{-N}))_{|x=\frac{R}{\lambda}},&
|\lambda| < R, \\
\frac{d^{j-1}}{dx^{j-1}}(x^{j-1} G(x^{N}))_{|x=\frac{R}{\lambda}},&
|\lambda| > R.
\end{array}
\right.
\end{equation}
In particular,
\begin{equation} \label{err12}
E_N((z-\lambda)^{-j}) =
\left\{ \begin{matrix}
\mathcal{O}\left(|\lambda|^{-j} \left( \frac{|\lambda|}{R} \right)^{N-j+1}
\right),&
|\lambda| < R , \\
\mathcal{O}\left(|\lambda|^{-j} \left( \frac{R}{|\lambda|} \right)^{N+j-1}
\right),& |\lambda| > R.
\end{matrix}
\right.
\end{equation} 
\end{lemma}
\begin{remark} If $f\in H(A(a_-,a_+),\mathbb{C})$ is meromorphic
on an open neighborhood of the closed annulus $A(a_-,a_+)^c$, then
the estimate \eqref{err6} can be sharpened as follows
\[ E_N(f)= \mathcal{O}(a_+^{-N}+ a_-^{-N}). \]
 In order to
see this, first consider the singular part that belongs to poles on
the boundary of $A(a_-,a_+)$, and use Lemma \ref{propsing}. Then apply
Theorem \ref{thecircle} to the remaining part on a slightly larger annulus.
\end{remark}

Consider a general contour $\Gamma$ in $\Omega$ with $2 \pi$-periodic
parametrization $\varphi(t),t\in[0,2\pi]$. Moreover, assume that
$\varphi$ has a $2 \pi$-periodic holomorphic extension to a strip
\begin{equation} \label{err16}
\varphi \in H(S(d_-,d_+),\Omega), \quad 
\varphi(z+2\pi) = \varphi(z).
\end{equation}
For definiteness, we also assume that
\begin{equation} \label{err17}
\varphi(z) \left\{ \begin{array}{rl}
               \in \mathrm{int}(\Gamma), &  0 < \im(z) < d_+, \\
                \notin \mathrm{int}(\Gamma), &  -d_- < \im(z) < 0.
                   \end{array}
            \right.
\end{equation}
Common examples are circles $\varphi(z)= z_0 + R e^{iz}$
with $z \in \mathbb{C}$ and
ellipses $ \varphi(z)= a \cos(z) + b \sin(z)$ with 
$ |\im(z)| < \mathrm{artanh}(\min(\frac{a}{b},\frac{b}{a}))$.

Let  $g \in H(\Omega,\mathbb{C})$, then the error of the trapezoid sum
for $f(z)=g(\varphi(z))\varphi'(z)$, $z\in S(d_-,d_+)$ is 
\begin{equation}\label{err18}
E_N(g) = \frac{1}{2\pi i} \int_{\Gamma} g(z) dz -
\frac{1}{iN} \sum_{j=0}^{N-1}g(\varphi(\frac{2 \pi j}{N}))\varphi'(\frac{2 \pi j}{N}) .
\end{equation}
From Theorem \ref{thestrip} we obtain an estimate
\begin{equation} \label{err19} 
|E_N(g)| \le \Phi(r_+)  G(e^{-N r_+})+
             \Phi(r_-) G(e^{-N r_-}),
\end{equation}
where $0 < r_- < d_-,0 <r_+ < d_+$ and 
$\Phi(r)= \max_{\im(z)=r}|\varphi'(z)||g(\varphi(z))|$.
The following lemma gives a rough estimate of the right-hand sides
for the pole function $g(z)=(z-\lambda)^{-j}, \lambda \in \Omega$.

\begin{lemma} \label{lemerr}
Let $\Omega$ be bounded and let $\varphi$ satisfy conditions \eqref{err16},
\eqref{err17}. Then there exist constants $C_1,C_2,C_3 > 0$ (depending on
$\varphi$, $j$ but not on $N$ or $\lambda\in \Omega$) such that
for $\mathrm{dist}(\lambda,\Gamma) \le C_3$,
\begin{equation} \label{err20}
|E_N((\cdot-\lambda)^{-j})| \le C_1 \mathrm{dist}(\lambda,\Gamma)^{-j}
\exp\left(-C_2 N \mathrm{dist}(\lambda,\Gamma)\right).
\end{equation}
\end{lemma}
\begin{proof} For a fixed $0 < q <1$  there are bounds
$|\varphi'(z)| \le M_+$ for $0 \le \im(z) \le qd_+$
and $|\varphi'(z)| \le M_-$ for $0 \le - \im(z) \le qd_{-}$.
Let $C_3= \max(M_+ d_+, M_-d_-)$ and define 
$r_+=\frac{q \mathrm{dist}(\lambda,\Gamma)}{M_+}$. Then there exists
some $z_+=s_+ + i r_+, 0\le s_+ < 2 \pi$ such that
\begin{eqnarray*} \min_{\im(z)=r_+}|\lambda -\varphi(z)| = &
       |\lambda-\varphi(z_+)| \ge |\lambda-\varphi(s_+)|
-|\varphi(s_+)-\varphi(z_+)| \\
\ge & \mathrm{dist}(\lambda,\Gamma) -M_+ r_+ = 
(1-q)\mathrm{dist}(\lambda,\Gamma).
\end{eqnarray*} 
The first term in \eqref{err19} can be estimated as follows
\begin{eqnarray*}
|\Phi(r_+)| G(e^{-Nr_+}) \le &  
M_+ \max_{\im z = r_+}|(\varphi(z)-\lambda)^{-j}| G(e^{-N r_+}) \\
\le & C(1-q)^{-j} M_+ \mathrm{dist}(\lambda,\Gamma)^{-j} 
\exp\left(-N \mathrm{dist}(\lambda,\Gamma) \frac{q}{M_+}\right). 
\end{eqnarray*}
The second term is treated analogously.
\end{proof}
As a consequence of Lemmas \ref{propsing} and \ref{lemerr} we obtain
an exponential estimate for the errors in \eqref{trap0} and \eqref{trap1}.
\begin{theorem} \label{interr}
Let $T\in H(\Omega,\mathbb{C})$ have maximum order $\kappa$ of poles
for the inverse in $\Omega$, cf. Theorem \ref{theinv}.
Further, let $\Gamma$ be a simple closed contour in $\Omega$ with
$\sigma(T)\cap \Gamma = \emptyset$ and such that the parametrization
$\varphi$ satisfies \eqref{err16} and \eqref{err17}.
Then there exist constants $C_1,C_2>0$ (depending on $T$ and $\hat{V}$
but not on $N$) such that the matrices from \eqref{trap0},\eqref{trap1}
satisfy
\begin{equation*} \label{matrixerr}
||A_p -A_{p,N} || \le C_1 d(T)^{-\kappa} e^{-C_2 N d(T)}, \quad p=0,1,
\end{equation*}
where $d(T)= \min_{\lambda \in \sigma(T)} \mathrm{dist}(\lambda,\Gamma)$
and $d(T)=1$ if $\sigma(T)=\emptyset$.
If $\Gamma$ is a circle with parametrization $\varphi(t)=z_0+Re^{it}$,
then the following estimate holds
\begin{equation*} \label{matrixcircle}
||A_p -A_{p,N} || \le C_1\left[\rho_-^{N-\kappa+1} +
\rho_+^{N+\kappa-1} \right], \quad p=0,1,
\end{equation*} 
where  
\begin{equation*}
\rho_-=\max_{\lambda\in \sigma(T),|\lambda-z_0| < R}
\frac{|\lambda-z_0|}{R}, \quad
\rho_+=\max_{\lambda\in \sigma(T),|\lambda-z_0| > R}
\frac{R}{|\lambda-z_0|}.
\end{equation*}
\end{theorem} 
Combining these estimates with the well-known perturbation theory
for singular value decompositions \cite{stsu90} we find that the
integral algorithm detects the correct rank $k$ of $A_{0,N}$ if
$N$ is sufficiently large. Further, the perturbation theory
 for simple eigenvalues \cite{stsu90} leads to the following corollary.
\begin{corollary} \label{corfinal}
Let the assumptions of Theorem \ref{alg1} and of Theorem \ref{interr} be
satisfied. Let $\lambda_1,\ldots,\lambda_k$ be the eigenvalues of
$T$ inside $\Gamma$ and let $\lambda_{1,N},\ldots,\lambda_{k,N}$ be the
eigenvalues from step 6 of the integral algorithm. With the notation
from Theorem \ref{interr} we then have the error estimates
\begin{equation*} \label{lambdaerr}
\max_{j=1,\ldots,n(\Gamma)}|\lambda_j - \lambda_{j,N}| \le
 C_1 d(T)^{-\kappa} e^{-C_2 N d(T)},
\end{equation*} 
in case of a general curve satisfying \eqref{err16},\eqref{err17},
and
\begin{equation*} \label{errcircle}
\max_{j=1,\ldots,n(\Gamma)}|\lambda_j - \lambda_{j,N}| \le
C\left[\rho_-^{N-\kappa+1} +
\rho_+^{N+\kappa-1} \right]
\end{equation*}
in case of a circle with radius $R$ and center $z_0$.
\end{corollary}

\subsection{Numerical examples}
\label{sec4.2}
\begin{example} \label{example1}
For the first test we choose a real quadratic polynomial
\begin{equation} \label{ex1}
T(z)= T_0 + z T_1 + z^2 T_2, \quad T_j \in \mathbb{R}^{60,60},j=0,1,2,
\end{equation}
where $T_0,T_1,T_2$ are taken at random  ({\it rand} from
MATLAB). In this case we can compare with the spectrum 
$\sigma_{\mathrm{polyeig}}$ resulting from MATLAB's {\it polyeig}.

Figure \ref{fig1}(left) shows the result from polyeig (open circles) and the
eigenvalues from Integral algorithm 1 (filled boxes) for the data
\begin{equation} \label{ex2}
\varphi(t)=R e^{it},\; t\in [0,2 \pi]\;, \;R=0.33,\; \mathrm{tol}_{\mathrm{rank}}=10^{-4},\;
\mathrm{tol}_{\mathrm{res}}= 10^{-1}.
\end{equation}
The eight eigenvalues inside the circle are detected and well approximated
by the integral algorithm. Figure \ref{fig1} (right) shows the errors
\begin{equation*} \label{errpoly}
e(\lambda_j)= \min\{ | \lambda_j -\mu| : \mu \in \sigma_{\mathrm{polyeig}}\}
\end{equation*}
for two characteristic eigenvalues inside the circle. Both show exponential
decay with respect to $N$ at approximately the same rate.

\begin{figure}[h] 
\psfrag{real}{$\re$}
\psfrag{imag}{$\im$}
\psfrag{N}{$N$}
\psfrag{logError}{$e(\lambda_j)$}
\centering
\includegraphics[width=0.48\textwidth,height=0.223\textheight]{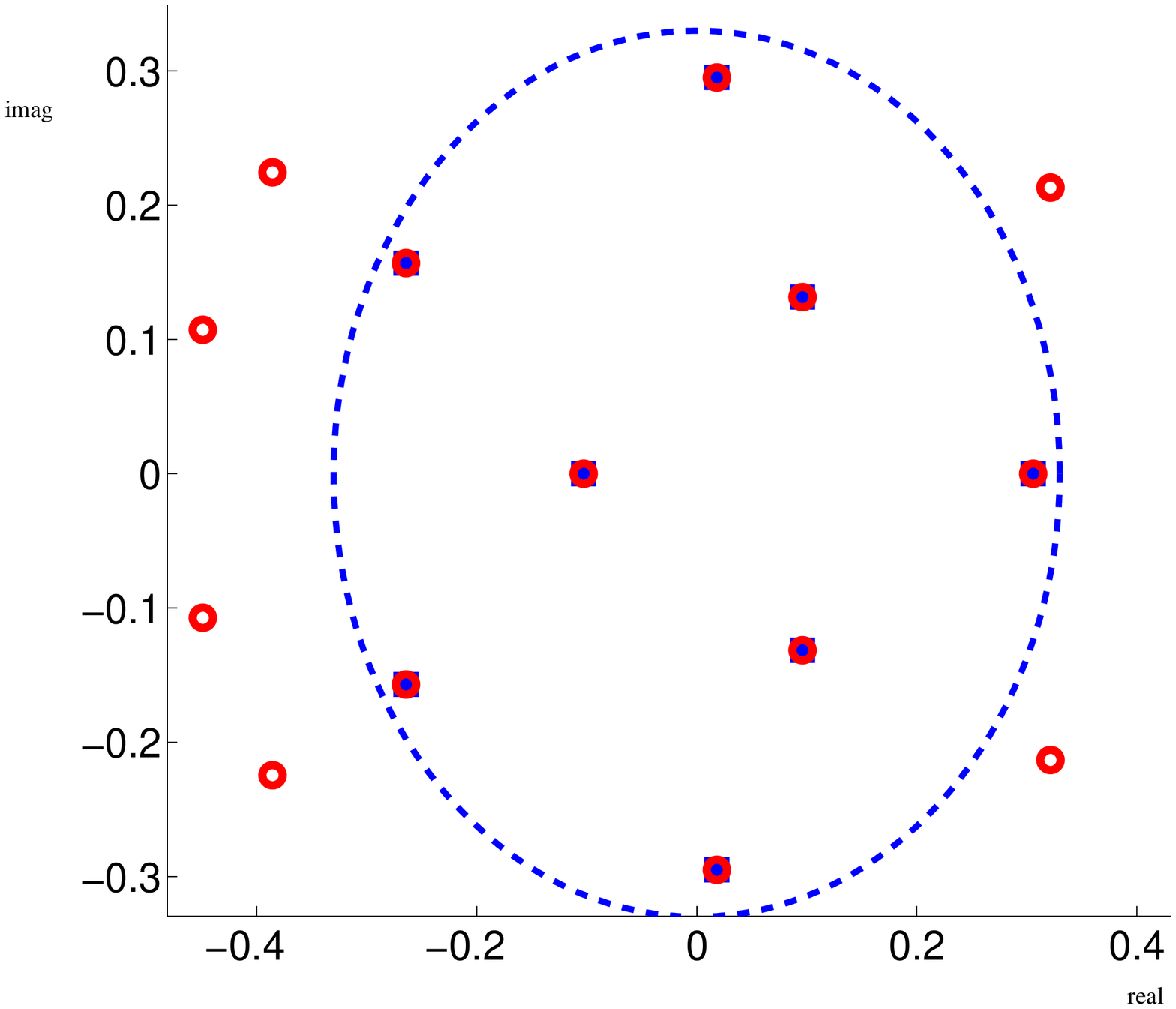} \quad
\includegraphics[width=0.48\textwidth]{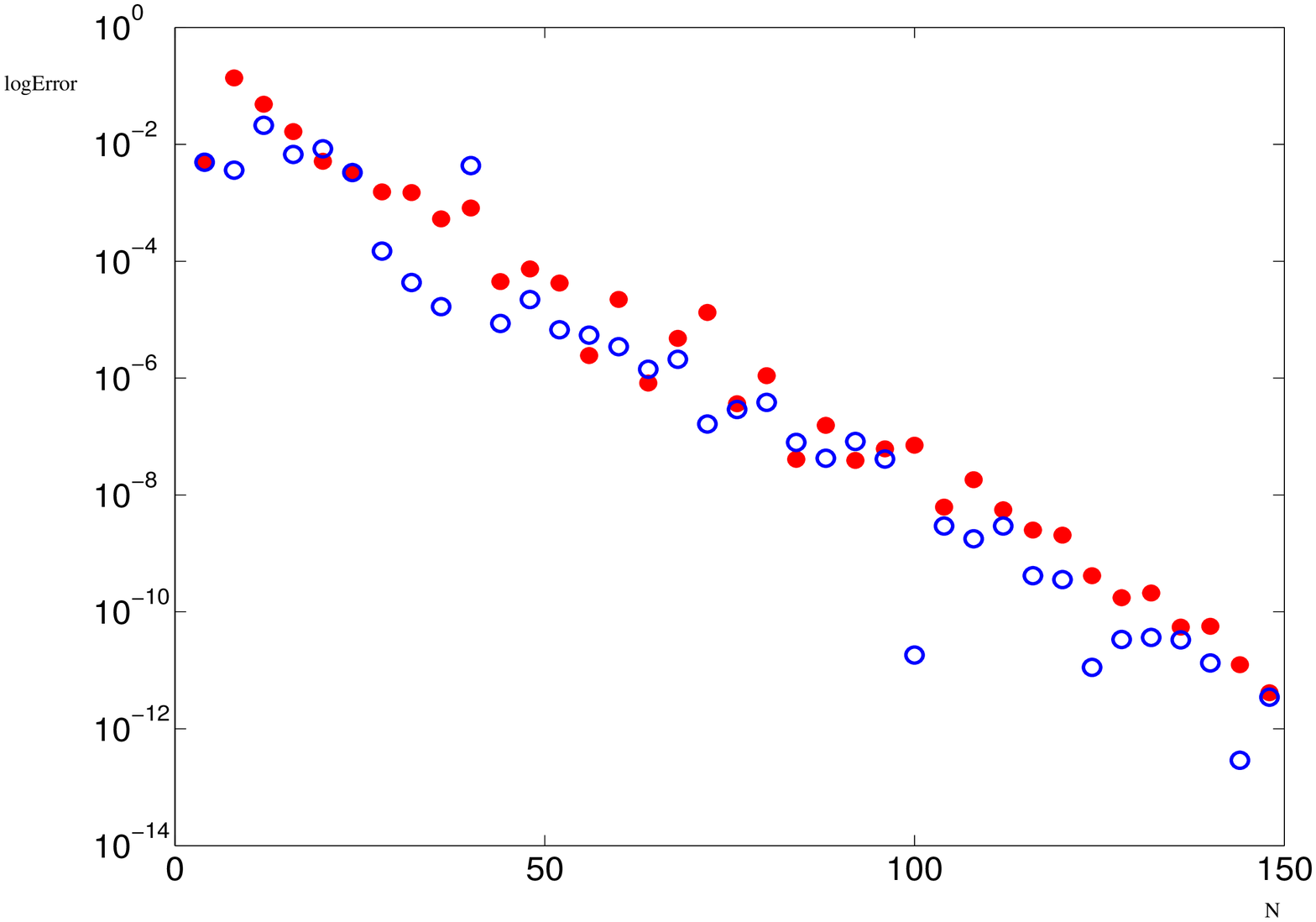}
\caption{\label{fig1} Example \ref{example1}.Eigenvalues of a quadratic eigenvalue problem  from
polyeig (open circles) and Integral algorithm 1 (filled squares) with $N=150$
(left). Difference $e(\lambda_j)$ of eigenvalues $\lambda_1 \approx0.30578$ (filled circles)
and
$\lambda_2\approx 0.0961 - 0.1315i$ (open circles) between polyeig and
the integral algorithm versus the number of nodes $N$ (right). }
  \end{figure}

\begin{figure}[h] 
\centering
\psfrag{N}{$N$}
\psfrag{logSingular}{$\sigma_j$}
\includegraphics[width=0.48\textwidth,height=0.25\textheight]{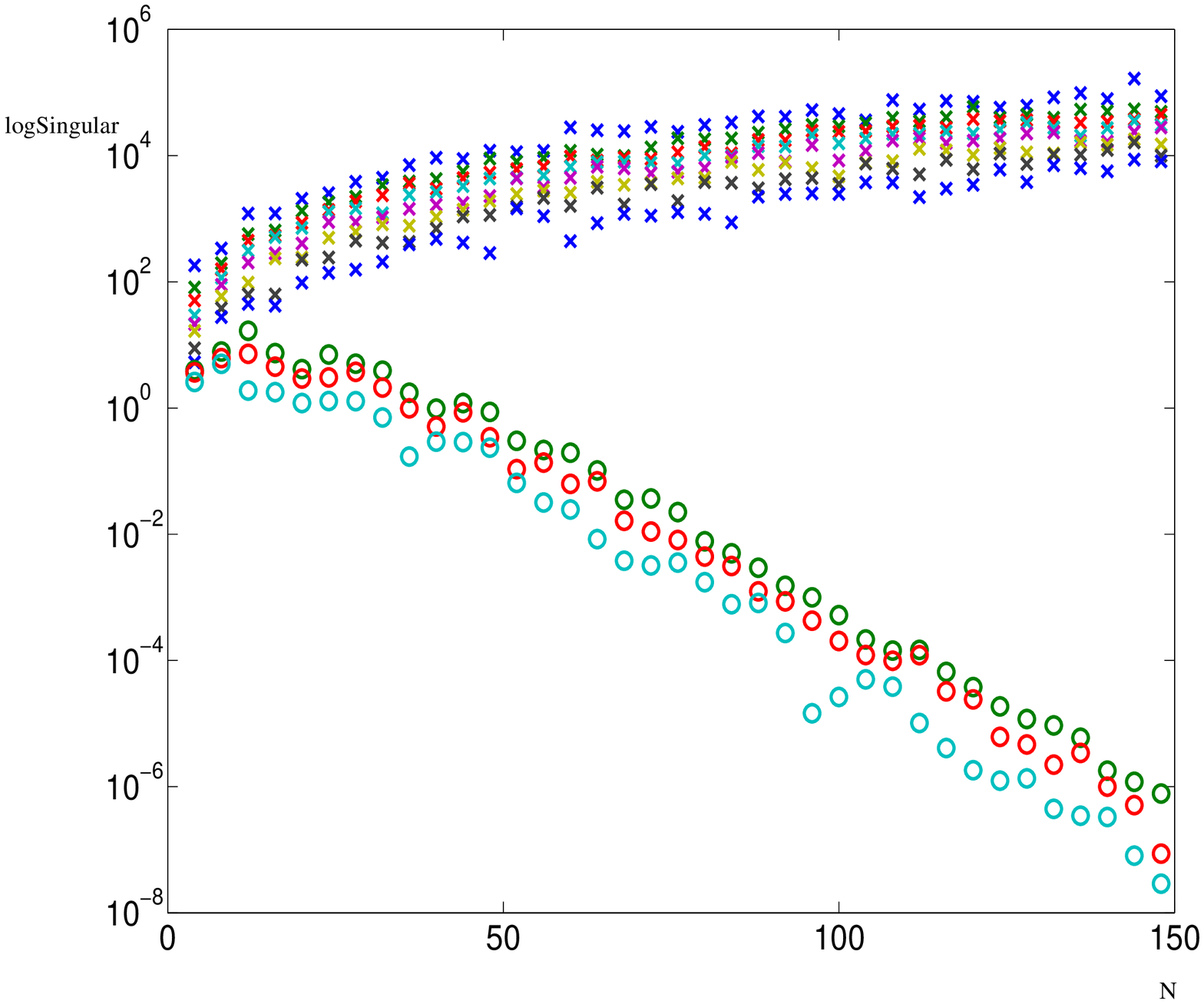} \quad
\includegraphics[width=0.48\textwidth]{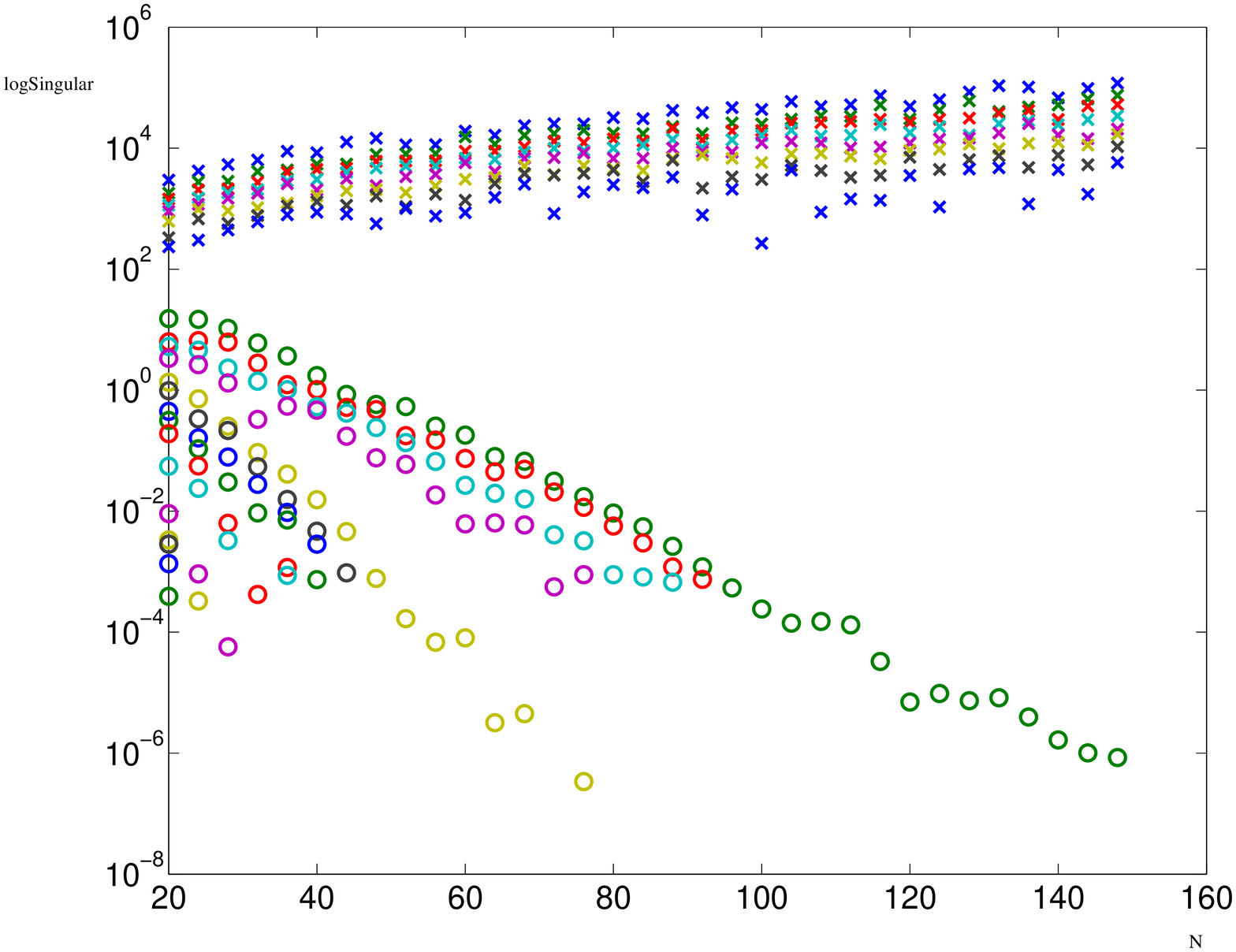}
\caption{\label{fig2} Example \ref{example1}. Singular values versus $N$ for a fixed number
of $l=11$ columns
in the integral algorithm (left), reduction of the number of singular values
by the rank test of the adaptive algorithm versus $N$ (right). }
\end{figure}

While Figure \ref{fig1} (left) results from the integral algorithm with
an adaptive number $l$ of columns (which yields $l=8$ at $N=150$), 
the computations in Figure \ref{fig1}(right) are done with a fixed  
number of $l=11$ columns. For this case
we show the behavior of the $11$ largest singular values of $A_{0,N}$
in Figure \ref{fig2} (left). Sufficient separation of singular values
already occurs at values $N\approx 25$, much smaller than $150$.
Figure \ref{fig2} (right) shows how the adaptive algorithm reduces the
number of singular values from $l=23$ at $N=20$ to $l=8$ for $N\ge 95$.
\end{example}
\begin{example} \label{example2}
For the next experiment we take random complex entries in  \eqref{ex1},
a fixed number $l=10$ of columns, and the same circle as in \eqref{ex2}.
Again, the $6$ eigenvalues inside the circle from {\it polyeig} are well 
approximated by the integral algorithm, see Figure \ref{fig3} (left).

\begin{figure}[h] 
\centering
\psfrag{real}{$\re$}
	\psfrag{imag}{$\im$}
\psfrag{N}{$N$}
	\psfrag{sigma}{$\sigma_j$}
\includegraphics[width=0.48\textwidth]{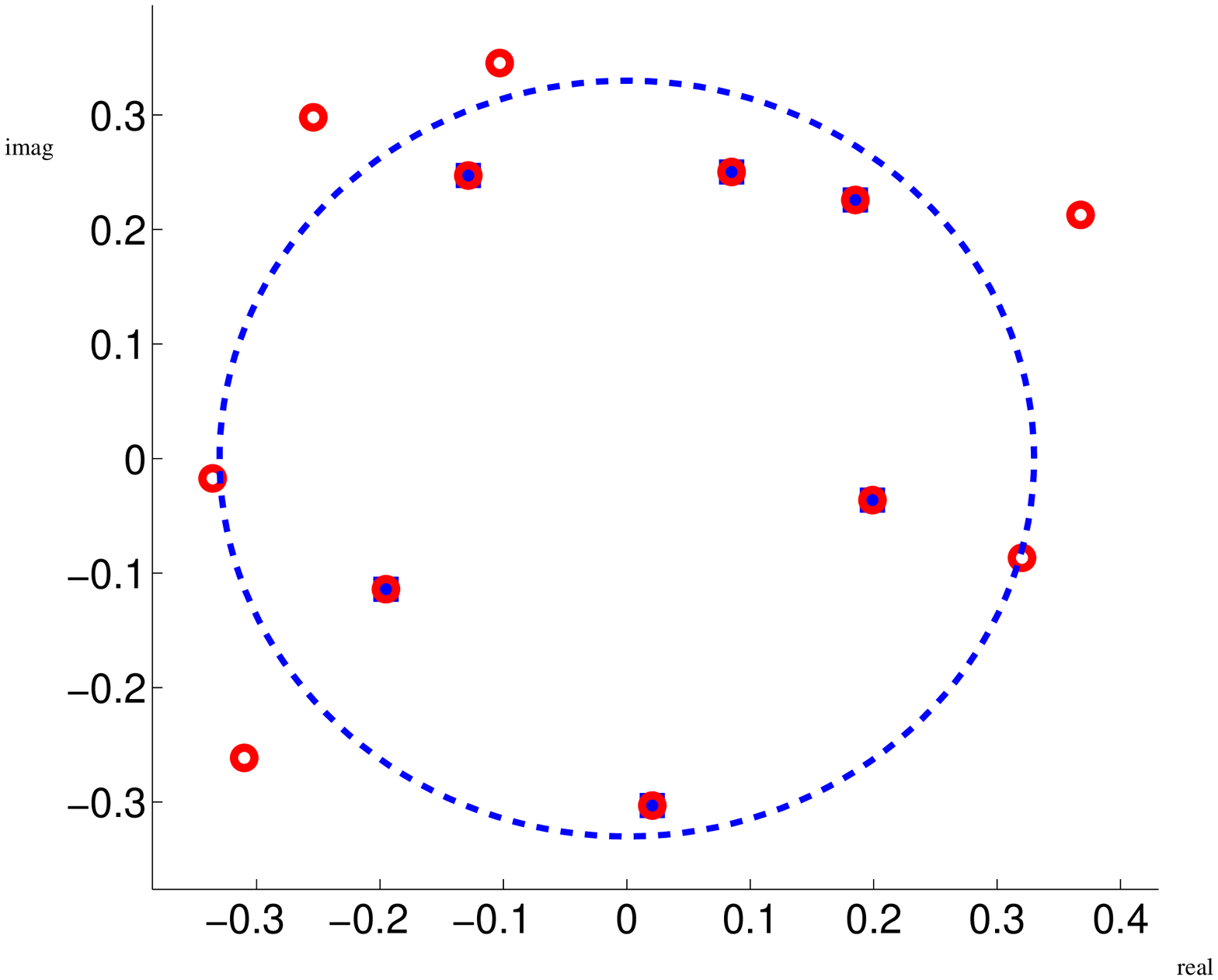} \quad
\includegraphics[width=0.48\textwidth,height=0.27\textheight]{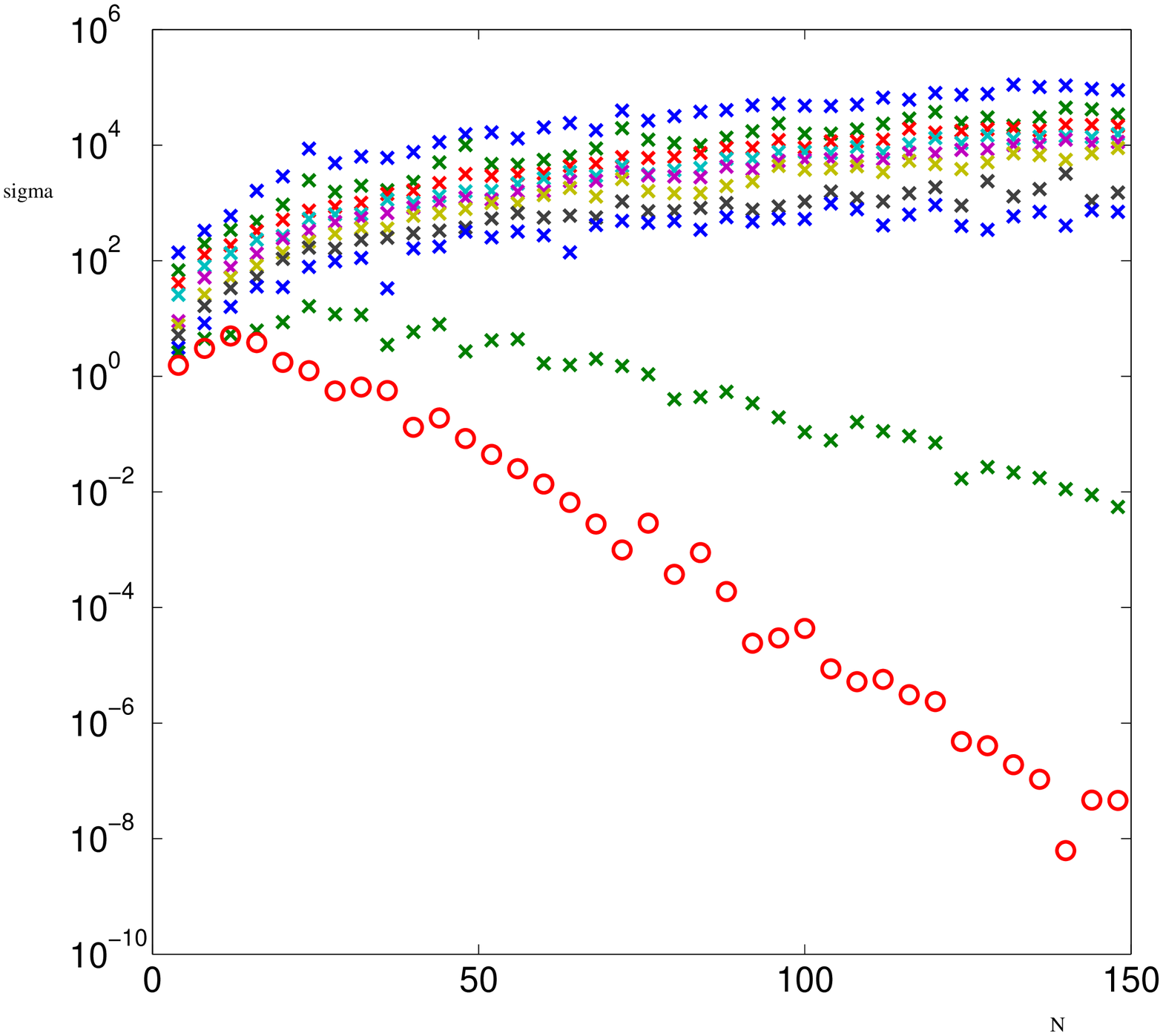}
\caption{\label{fig3} Example \ref{example2}. Eigenvalues from polyeig (open circles) and eigenvalues from
the integral algorithm for a random quadratic complex matrix polynomial
(left), singular values of integral algorithm with $l=10$ columns versus 
the number $N$ of quadrature nodes for
the same example (right). }
\end{figure}
But this time the singular values do not separate as well as in Figure 
\ref{fig2}
(left). Two of them decay rather slowly, while two others, due to 
eigenvalues very close but outside the contour, remain of order one.
However, this behavior does not result in spurious eigenvalues. On the 
contrary, if we keep $l=10$ for the eigenvalue computation, then this
yields the $6$ eigenvalues inside and in addition 
the four eigenvalues lying closest to the contour, but outside.
Such a behavior is also suggested by our error analyis in Section \ref{sec4.1}
according to which the principle error term depends on the distance
of eigenvalues to the contour, both for eigenvalues inside and outside.
Computational experience shows that only very small singular values
($\approx 10^{-10}$) lead to spurious eigenvalues and these can
be easily avoided by the residual test in Step 6.
\end{example}
\begin{example} \label{example3}
This example, taken from \cite{so06} and \cite{kr09}, is a finite
element discretization of a nonlinear boundary eigenvalue problem
\begin{equation*} - u''(x)= \lambda u(x), 0 \le x \le 1,
          u(0)=0=u'(1)+ \frac{\lambda}{\lambda-1}u(1).
\end{equation*}
The matrix function is $T(z)= T_1 + \frac{1}{1-z}e_m e_m^T - z T_3$, where 
\begin{equation*} \label{fem}
T_1 = m \begin{pmatrix} 2 & -1 & & \\
                    -1 &  \ddots & \ddots & \\
                       & \ddots & 2 & -1 \\
                       & & -1 & 1 
        \end{pmatrix}, \quad
T_3 = \frac{1}{6m} \begin{pmatrix}
               4 & 1 & &  \\
               1 &\ddots & \ddots & \\
                 & \ddots & 4 & 1 \\
                 &  & 1 & 2  
               \end{pmatrix}.
\end{equation*}

We use $m=400$ and compute five eigenvalues in the interval
$[2,298]$.
Again Figure \ref{fig4a} (left) shows the real eigenvalues in the circle
which agree with those from \cite{kr09}.  Note that we avoided the singularity
of $T$ at $z=1$. The residuals of the computed eigenvectors and eigenvalues
decay exponentially as expected, see Figure \ref{fig4a}, but not
as smooth as in the previous examples.

\begin{figure}[h] 
\centering
\psfrag{real}{$\re$}
	\psfrag{imag}{$\im$}
\psfrag{N}{$N$}
\psfrag{res}{$res(\lambda_j)$}
\includegraphics[width=0.48\textwidth]{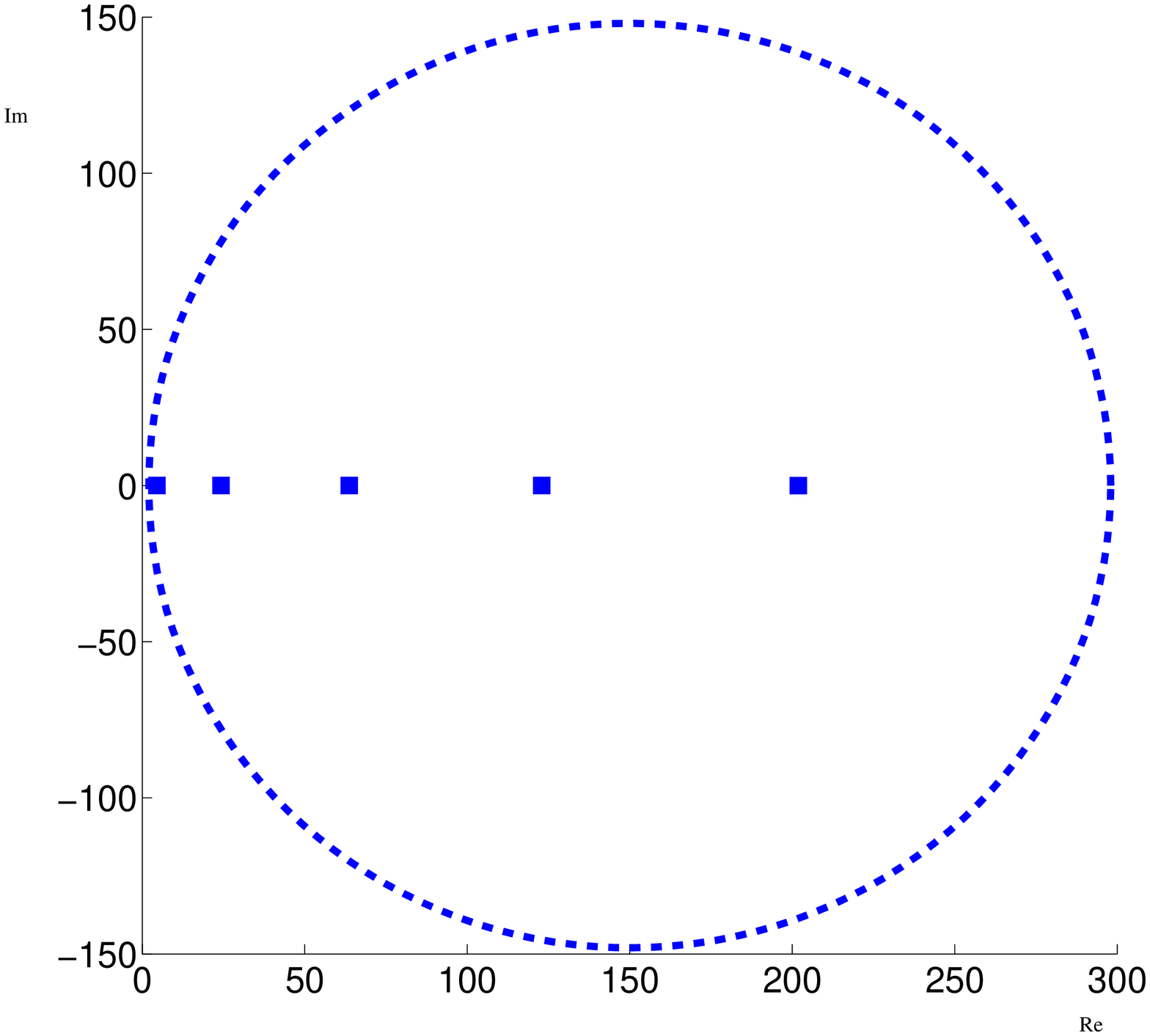} \quad
\includegraphics[width=0.48\textwidth]{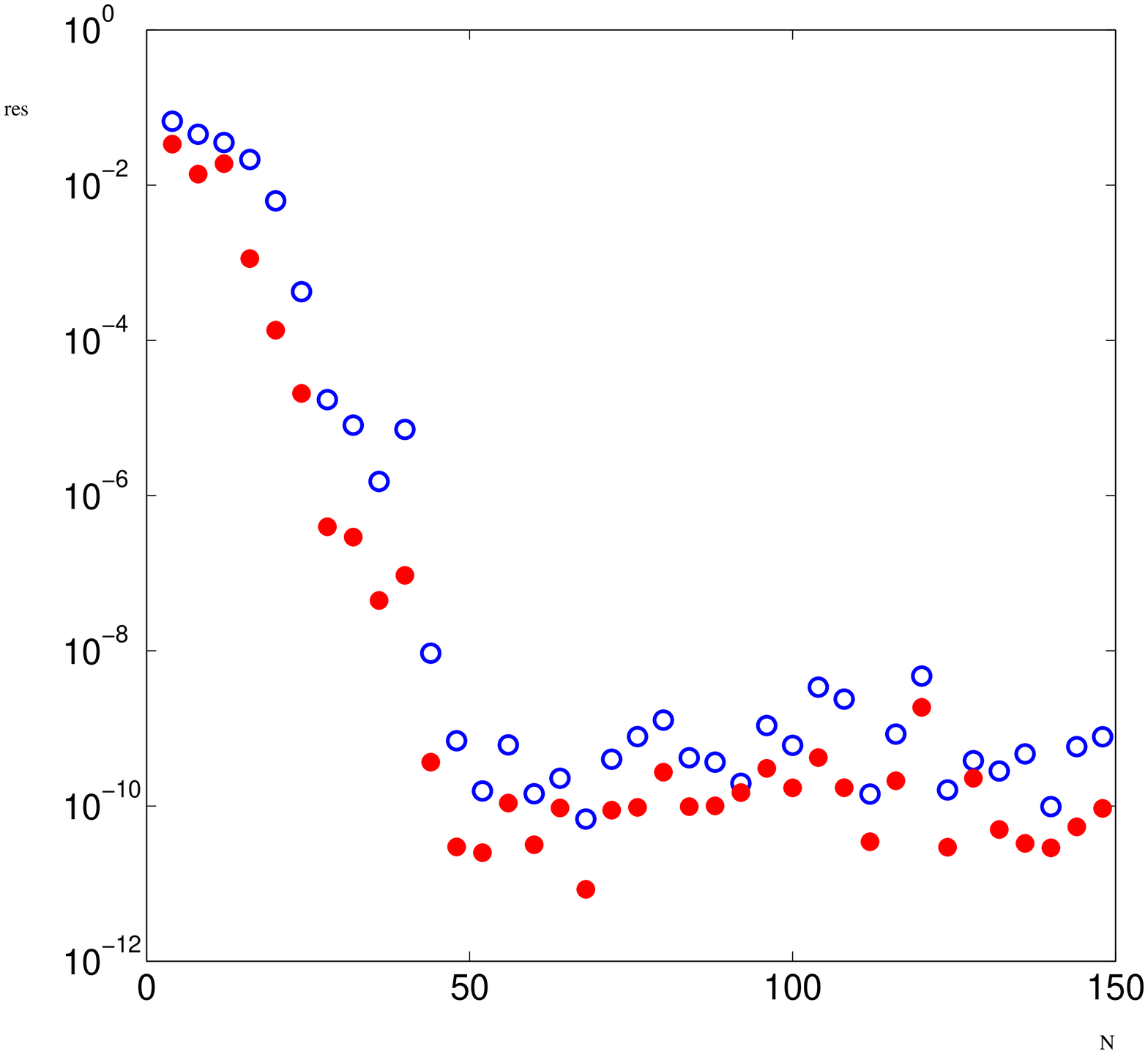}
\caption{\label{fig4a} Example \ref{example3}. Eigenvalues  from
the integral algorithm for the finite element discretization of a nonlinear
boundary eigenvalue problem (left), 
decay of residuals $\mathrm{res}(\lambda_j)=||T(\lambda_j)(v_j)||$ for 
$\lambda_1\approx 24$ (open circles), $\lambda_2\approx 123$ (filled circles)
 versus the number $N$ of quadrature nodes for
the same example (right). }
\end{figure}

\end{example}

\clearpage

\begin{example} \label{example4}
Consider the quadratic polynomial
\begin{equation} \label{excrit}
T(z)= T_0 + (z-a)(b-z)T_1, \quad a<b\in \mathbb{R}, \quad T_0,T_1\in
\mathbb{R}^{15,15},
\end{equation}
where $T_0$ has zeroes in the first column. All other entries
of $T_0,T_1$ are chosen at random. Then $T(z)$ has different eigenvalues $a$ and
$b$ with the same eigenvector $e^1\in \mathbb{R}^{m}$. This is a critical case
since  the rank condition \eqref{Vcond} is violated. In Figure \ref{fig4} (left)
we show the results of {\it polyeig} and of the integral algorithm (with $l=5$
and the data from \eqref{ex2}). There are three eigenvalues inside the circle.
Both eigenvalues $a=-0.2$ and $b=0.1$ are missed by the integral method,
while the third one is found, though at lower accuracy than in the 
previous examples. Figure \ref{fig4} shows that only one singular value
stays of order one when $N$ is increased. This example will be reconsidered
in Section \ref{sec5}.

\begin{figure}[h] 
\centering
\psfrag{real}{$\re$}
	\psfrag{imag}{$\im$}
\psfrag{N}{$N$}
	\psfrag{sigma}{$\sigma_j$}
\includegraphics[width=0.48\textwidth]{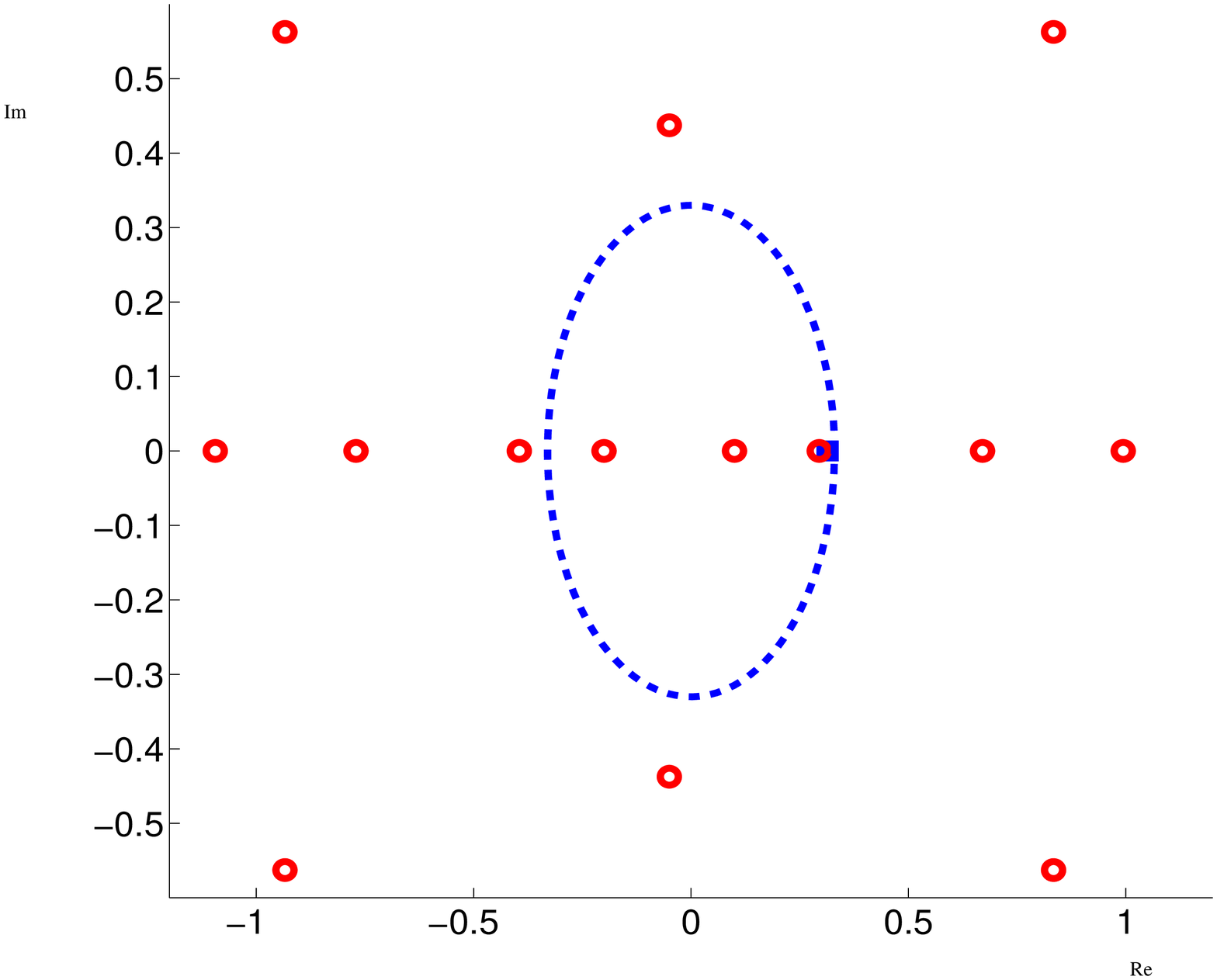} \quad
\includegraphics[width=0.48\textwidth,height=0.26\textheight]{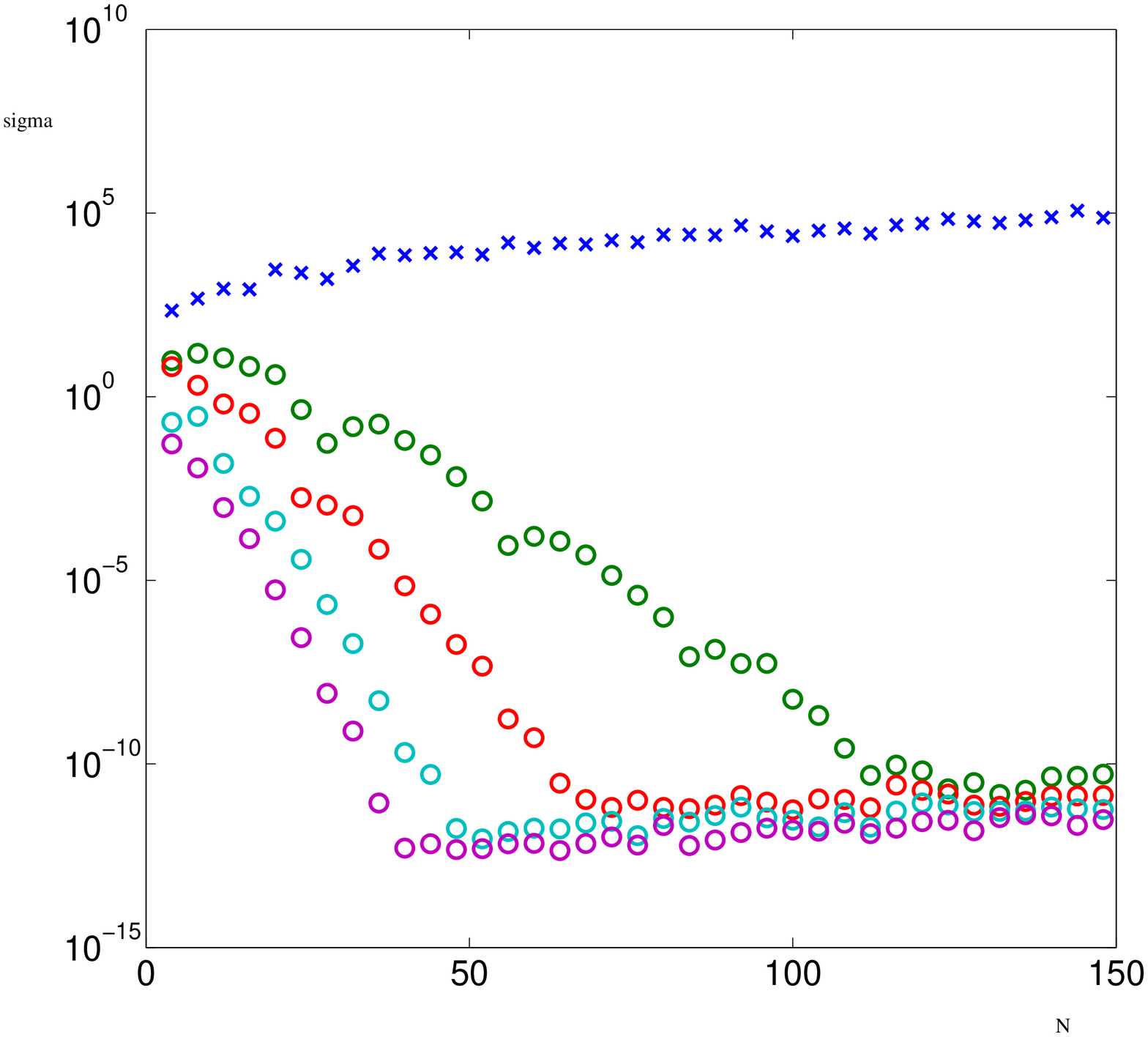}
\caption{\label{fig4} Example \ref{example3}. Eigenvalues from polyeig (open circles) and eigenvalues from
the integral algorithm for a quadratic matrix polynomial with rank defect
(left), singular values of integral algorithm with $l=5$ columns versus 
the number $N$ of quadrature nodes for
the same example (right). }
\end{figure}
 
\end{example}

\section{The algorithm for many eigenvalues}
\label{sec5}
In this section we show how the method from Section \ref{sec3}
can be extended to 
nonlinear eigenvalue problems with more eigenvalues than the dimension 
of the system, i.e. $m < k$, and to the rank deficient cases, see
Remark \ref{rem3.2} and Example \ref{example4}.
\subsection{Construction of algorithm} \label{sec5.1}
In case $m <k$ condition \eqref{Vcond} is always violated and there is
no matrix $\hat{V}$ satisfying \eqref{Wcond}.
Therefore, we compute more integrals of type (\ref{A0}),(\ref{A1}), namely
\begin{equation*} \label{Aq}
A_p = \frac{1}{2 \pi i} \int_{\Gamma} z^p T(z)^{-1}\hat{V} dz 
\in {\mathbb C}^{m,l},
\quad p \in {\mathbb N}.
\end{equation*}
Here we assume that $\hat{V}\in \mathbb{C}^{m,l}$ with $l \le m$.
In fact, in case $k>m$ we set $\hat{V}=I_m$ instead of making a random choice.

From Theorem \ref{intresolvent} we obtain
\begin{equation} \label{Aqexpress}
A_p= V \Lambda^p W^H \hat{V}, \;  
\; p \in {\mathbb N},
\end{equation}
where $V,W \in \mathbb{C}^{m,k}$ are given by \eqref{vdef} and \eqref{wdef}
and $\Lambda$ has the normal form \eqref{lambdanormal}. 

Now we choose $K \in {\mathbb N}, K\ge 1$ 
and form the $Km \times Kl$ matrices
\begin{equation} \label{Bform}
B_0= \begin{pmatrix} A_0 & \cdots & A_{K-1} \\
                           \vdots & & \vdots     \\
                          A_{K-1} & \cdots & A_{2K-2}
          \end{pmatrix}, \quad
B_1= \begin{pmatrix} A_1 & \cdots & A_{K} \\
                           \vdots & & \vdots     \\
                          A_{K} & \cdots & A_{2K-1}
          \end{pmatrix}.  
\end{equation}
From (\ref{Aqexpress}) we find the  representations
\begin{equation} \label{B0express}
B_0 = \begin{pmatrix} V \\ \vdots \\ V \Lambda^{K-1}
           \end{pmatrix}
           \begin{pmatrix} W^H\hat{V} & \cdots & \Lambda^{K-1}W^H \hat{V}
           \end{pmatrix},
\end{equation}
and
\begin{equation} \label{B1express}
B_1 = \begin{pmatrix} V \\ \vdots \\ V \Lambda^{K-1}
           \end{pmatrix} \Lambda
           \begin{pmatrix} W^H \hat{V} & \cdots & \Lambda^{K-1}W^H \hat{V}
           \end{pmatrix}.
\end{equation}

We assume that  $K$ has been chosen such that the following
rank condition holds
\begin{equation}\label{rankcond}
\text{rank}\begin{pmatrix} V \\ \vdots \\ V \Lambda^{K-1}
           \end{pmatrix} =k.
\end{equation}
The smallest index having this property is called the minimality index in
\cite{kr09}.
In case $k>m$ this can be expected to hold if we choose
 \begin{equation*} \label{Kselect}
 (K-1)m < k \le Km.
 \end{equation*}
In case $k\le m$ with $\mathrm{rank}(V)<k$ (see Remark \ref{rem3.2}(b))
the following lemma shows that \eqref{rankcond} holds for $K$ larger
than the sum of the maximal ranks at all eigenvalues.
\begin{lemma} \label{rankbound}
Let the assumptions of Corollary \ref{cor1} be satisfied.
Then the rank conditon \eqref{rankcond} holds with $k$ as defined
in \eqref{kdef} for
\begin{equation*} \label{Klarge}
K \ge \sum_{n=1}^{n(\mathcal{C})} \max_{1\le \ell \le L_n}m_{\ell,n}.
\end{equation*}
\end{lemma}
\begin{proof} Let $M_n=\max_{1\le \ell \le L_n}m_{\ell,n}$ and 
$M= \sum_{n=1}^{n(\mathcal{C})}M_n$. Assume that $V \Lambda^j x=0,j=0,\ldots,M-1$
for some $x \in \mathbb{C}^{m}$. For any $n\in \{1,\ldots,n(\mathcal{C})\}$
and $0 \le \beta \le M_n -1$ consider the polynomial
\begin{equation*} \label{Lagrange}
P_{n,\beta}(z) = (z-\lambda_n)^{\beta} 
\prod_{r=1, r\neq n}^{n(\mathcal{C})}(z-\lambda_r)^{M_r}.
\end{equation*}
By our assumption $0 = V P_{n,\beta}(\Lambda)x$. We partition
according to \eqref{lambdanormal} 
\begin{eqnarray*} 
V =& \begin{pmatrix} V_1 & \cdots & V_{n(\mathcal{C})}
     \end{pmatrix},
V_n = \begin{pmatrix} V_{n,1} &\cdots & V_{n,L_n}
      \end{pmatrix}, 
V_{n,\ell} = \begin{pmatrix} v^{\ell,n}_{0} &\cdots & v^{\ell,n}_{m_{\ell,n}-1}
             \end{pmatrix} \\
x =& \begin{pmatrix} x_1 & \cdots & x_{n(\mathcal{C})}
     \end{pmatrix},
x_n = \begin{pmatrix} x_{1,n} &\cdots & x_{L_n,n}
      \end{pmatrix}, 
x_{\ell,n} = \begin{pmatrix} x^{\ell,n}_{0} &\cdots & x^{\ell,n}_{m_{\ell,n}-1}
             \end{pmatrix}.
\end{eqnarray*}
Using $(J_{\tilde{n}}-\lambda_{\tilde{n}})^{M_{\tilde{n}}}=0$ for $\tilde{n} \neq n$
we obtain
\begin{equation}\label{sumid}
 0 =  \sum_{\ell=1}^{L_n} V_{n,\ell} 
\prod_{\tilde{n}\neq n,  m_{\ell,n}-1 \ge \beta
}^{n(\mathcal{C})} (J_{n,\ell}-\lambda_{\tilde{n}})^{M_{\tilde{n}}}
(J_{n,\ell}-\lambda_n)^{\beta} x_{\ell,n}.
\end{equation}
From this we conclude by induction on $\beta=M_{n}-1,\ldots,0$ that
\begin{equation} \label{xvanish}
 x^{n,\ell}_{\nu}= 0, \quad \text{if} \; 
\beta\le \nu \le m_{\ell,n}-1.
\end{equation}
For $\beta=M_{n}-1$, equation \eqref{sumid} reads
\begin{equation*}
0= \prod_{\tilde{n}\neq n}^{n(\mathcal{C})} 
(\lambda_n-\lambda_{\tilde{n}})^{M_{\tilde{n}}}
\sum_{\ell=1, m_{\ell,n}=M_n}^{L_n} v_{0}^{\ell,n} x^{\ell,n}_{m_{\ell,n}-1},
\end{equation*}
and thus  \eqref{xvanish} holds for $\beta=M_n$ by the linear independence of 
the vectors $v_{0}^{\ell,n}$ (cf. Definition \ref{def2} (iv)).
For the induction step we use \eqref{sumid} with $\beta-1$ instead of $\beta$.
Together with \eqref{xvanish} we find
\begin{equation*}
 0= \prod_{\tilde{n}\neq n}^{n(\mathcal{C})} 
(\lambda_n-\lambda_{\tilde{n}})^{M_{\tilde{n}}}
\sum_{\ell=1, m_{\ell,n}\ge \beta}^{L_n} v_{0}^{\ell,n} x^{\ell,n}_{\beta-1},
\end{equation*}
which shows that \eqref{xvanish} holds for $\beta -1$.
Thus we have shown $x=0$ and this finishes the proof.
\end{proof}

The computational procedure is now a straightforward generalization
of Section \ref{sec3.1}.
First compute $B_0,B_1 \in \mathbb{C}^{Km,Kl}$ from \eqref{Bform}.
In addition to \eqref{rankcond}, assume
\begin{equation} \label{rankw}
\mathrm{rank} \begin{pmatrix} W^H\hat{V} & \cdots & \Lambda^{K-1}W^H \hat{V}
           \end{pmatrix} = k .
\end{equation}
Let us abbreviate
\begin{equation*}\label{vwabbr}
V_{[K]}=\begin{pmatrix} V \\ \vdots \\ V \Lambda^{K-1}
           \end{pmatrix} \in \mathbb{C}^{Km,k}, \quad 
W_{[K]}^H =\begin{pmatrix} W^H\hat{V} & \cdots & \Lambda^{K-1}W^H \hat{V}
           \end{pmatrix} \in \mathbb{C}^{k,Kl}.
\end{equation*}
Compute the SVD 
\begin{equation*} \label{B0svd}
V_{[K]} W_{[K]}^H = B_0 = V_0 \Sigma_0 W_0^H,
\end{equation*}
where 
$V_0\in \mathbb{C}^{Km,k}, V_0^H V_0=I_k$, 
$ \Sigma_0=\text{diag}(\sigma_{1},\ldots,\sigma_{k})\in \mathbb{C}^{k,k}$,
and  $W_0\in\mathbb{C}^{Kl,k}$, $W_0^HW_0=I_k$.
From the rank conditions \eqref{rankcond},\eqref{rankw},
\begin{equation*} \label{singvalK}
 \sigma_{1} \geq \ldots \sigma_{k} > 0= \sigma_{k+1}=  \ldots 
= \sigma_{Kl}.
\end{equation*}
The rank condition \eqref{rankcond} also implies
\begin{equation*}
R(B_0)=R(V_{[K]})=R(V_0).
\end{equation*}
Thus the matrix  $S=V_0^H V_{[K]} \in \mathbb{C}^{k,k}$ is nonsingular
 and satisfies
\begin{equation} \label{SrelK}
V_{[K]} = V_0 S.
\end{equation}
 With (\ref{B0express}), (\ref{SrelK}) we find
\begin{equation*}
W_{[K]}^H = S^{-1} \Sigma_0W_0^H,
\end{equation*} 
and then from \eqref{B1express}
\begin{equation*}
  B_1= V_{[K]} \Lambda W_{[K]} ^H  =
		V_0 S \Lambda S^{-1}\Sigma_0W_0^H.
\end{equation*}
Finally, this leads to
\begin{equation} \label{deq}
	D: = V_0^H B_1 W_0\Sigma_0^{-1} = S \Lambda S^{-1}.
\end{equation}
Therefore, the analog of Theorem \ref{algmult} is 
\begin{theorem} \label{algK}
Suppose that $T \in H(\Omega,\mathbb{C}^{m,m})$ has no eigenvalues on
the contour $\Gamma$ in $\Omega$ and pairwise distinct eigenvalues
$\lambda_n,n=1,\ldots,n(\Gamma)$ inside $\Gamma$ with partial
multiplicities $m_{1,n} \ge \ldots \ge m_{L_n,n}, n=1,\ldots,n(\Gamma)$.
Assume that the rank conditions \eqref{rankcond},\eqref{rankw} are satisfied
with  $k$ given by \eqref{kdef}.
 Then the matrix $D\in \mathbb{C}^{k,k}$ from \eqref{deq} has Jordan normal
form \eqref{lambdanormal} with the same eigenvalues $\lambda_n$
and partial multiplicities   $m_{\ell,n}$ 
($\ell=1,\ldots,L_n,n=1,\ldots,n(\Gamma)$). 
Suitable CSGEs for $T$ can be obtained from corresponding CSGEs 
$s_j^{\ell,n}$ for $D$ via
\begin{equation*} \label{gevK}
v_j^{\ell,n} = V_0^{[1]} s_j^{\ell,n}, \quad 0\le j \le m_{\ell,n}-1, 
1\le \ell \le L_n, 1\le n \le n(\Gamma),
\end{equation*}
where $V_0^{[1]}$ is the upper $m \times k$ block in 
\begin{equation} \label{v0rep}
V_0 = \begin{pmatrix} V_0^{[1]} \\ \vdots \\ V_0^{[K]} \end{pmatrix}.
\end{equation}
\end{theorem}
\begin{remark} In a sense this generalization is similar to
linearizing a polynomial eigenvalue problem by increasing the dimension.
Note, however, that this only becomes necessary if there are too many
eigenvalues inside the contour, or if rank defects occur that are
not present in linear eigenvalue problems.
\end{remark}
The generalization of the algorithm from Section \ref{sec3.3} is the following.
\bigskip

\noindent
{\bf Integral algorithm 2}
\smallskip

\noindent
{\bf Step 1:} Choose numbers $l \le m$, $K\ge 1$ and a matrix 
$\hat{V}\in \mathbb{C}^{m,l}$ at random. If  more than $m$
eigenvalues are expected inside $\Gamma$, let $l=m, \hat{V}=I_m$.
\smallskip

\noindent 
{\bf Step 2:} Compute
\begin{equation*} \label{trapK}
 A_{p,N} = \frac{1}{iN} \sum_{j=0}^{N-1}T(\varphi(t_j))^{-1} \hat{V}
\varphi(t_j)^p \varphi'(t_j), \quad p=0,\ldots,2K-1,
\end{equation*}
and form $B_{0,N}$,$B_{1,N}$ as in \eqref{Bform}.
\smallskip

\noindent
{\bf Step 3:} Compute the SVD $B_{0,N}=V \Sigma W^H$, where \\
$V\in \mathbb{C}^{Km,Kl}$, $W \in \mathbb{C}^{Kl,Kl}$, $V^HV=W^HW=I_{Kl}$, 
$\Sigma=\mathrm{diag}(\sigma_1,\sigma_2,\ldots,\sigma_{Kl})$.
\smallskip

\noindent
{\bf Step 4:} Perform a rank test for $\Sigma$, i.e.
 find $0<k\le Kl$ such that \\
$\sigma_1 \ge \ldots \ge \sigma_k > \sigma_{k+1} \approx \ldots
\approx \sigma_{Kl} \approx 0 $. \\
If $k=Kl$ then increase $l$ or $K$  and go to Step 1.\\
Else let $V_0=V(1:Km,1:k),W_0=W(1:Kl,1:k)$ and \\ 
$\Sigma_0=\mathrm{diag}(\sigma_1,\sigma_2,\ldots,\sigma_k)$.
\smallskip

\noindent
{\bf Step 5:} Compute $D=V_0^H B_{1,N} W_0 \Sigma_0^{-1} \in 
\mathbb{C}^{k,k}$.
\smallskip

\noindent
{\bf Step 6:} Solve the eigenvalue problem for $D$ \\
$DS =S \Lambda$, $S=(s_1 \ldots s_k), \Lambda=
\mathrm{diag}(\lambda_1,\ldots,\lambda_k)$.\\
If $||T(\lambda_j)v_j||$ is small and $\lambda_j \in \mathrm{int}(\Gamma)$
accept $v_j=V_0^{[1]} s_j$ (with $V_0^{[1]}$ from \eqref{v0rep})
 as eigenvector and $\lambda_j$ as eigenvalue. 

\subsection{Numerical Examples} \label{sec5.2}
\begin{example} \label{example5}
We  apply the integral algorithm 2 to the rank deficient example
\eqref{excrit}, where $K=2,l=3$ and  the contour is the circle from 
\eqref{ex2}. Now the eigenvalues $a=-0.2$ and $b=1$ are reproduced
correctly (see Figure \ref{fig5}(left)), and three singular values survive as
expected (Figure \ref{fig5} (right)).

\begin{figure}[h] 
\centering
\psfrag{real}{$\re$}
	\psfrag{imag}{$\im$}
\psfrag{N}{$N$}
	\psfrag{sigma}{$\sigma_j$}
\includegraphics[width=0.48\textwidth,height=0.2\textheight]{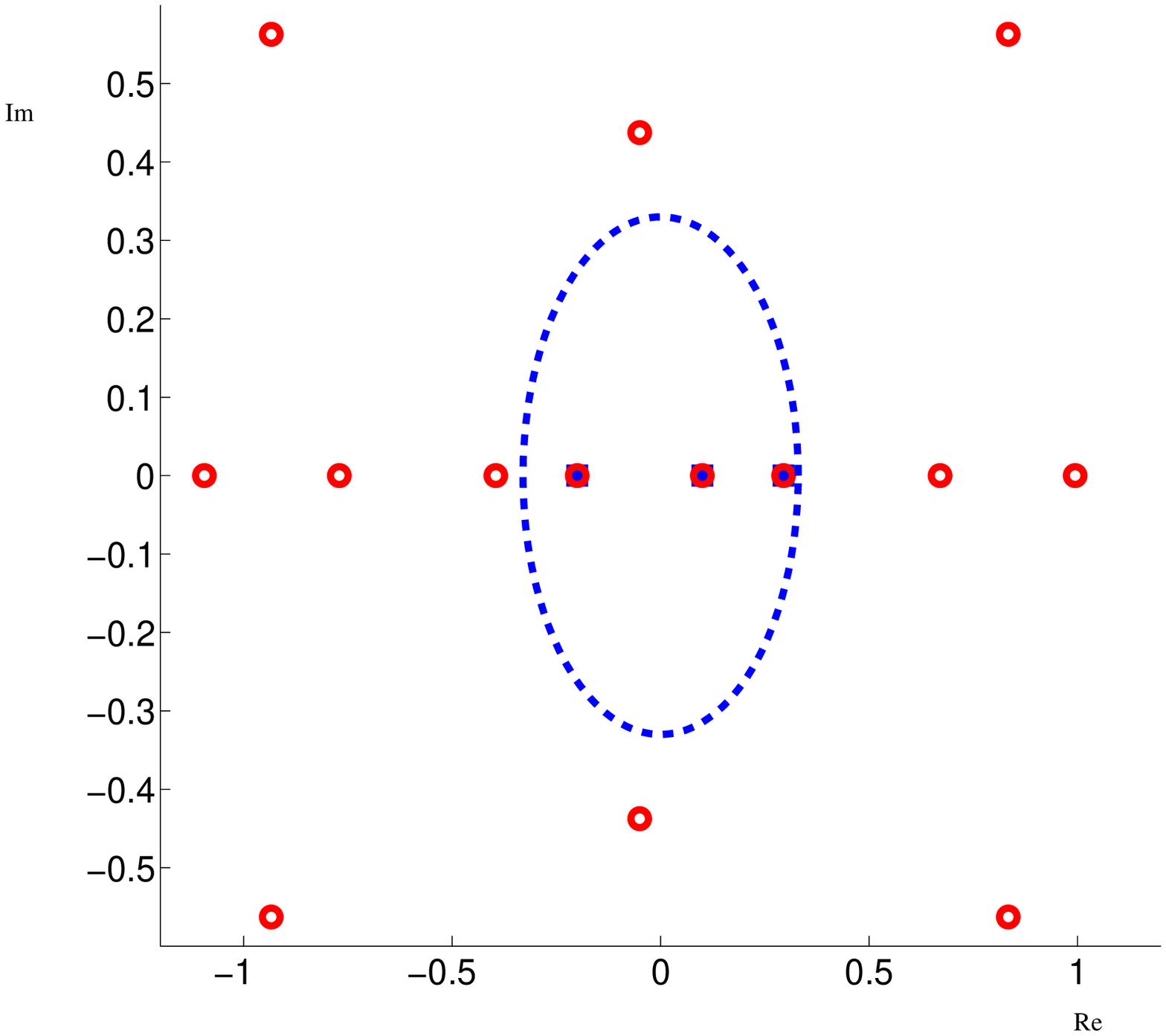} \quad
\includegraphics[width=0.48\textwidth]{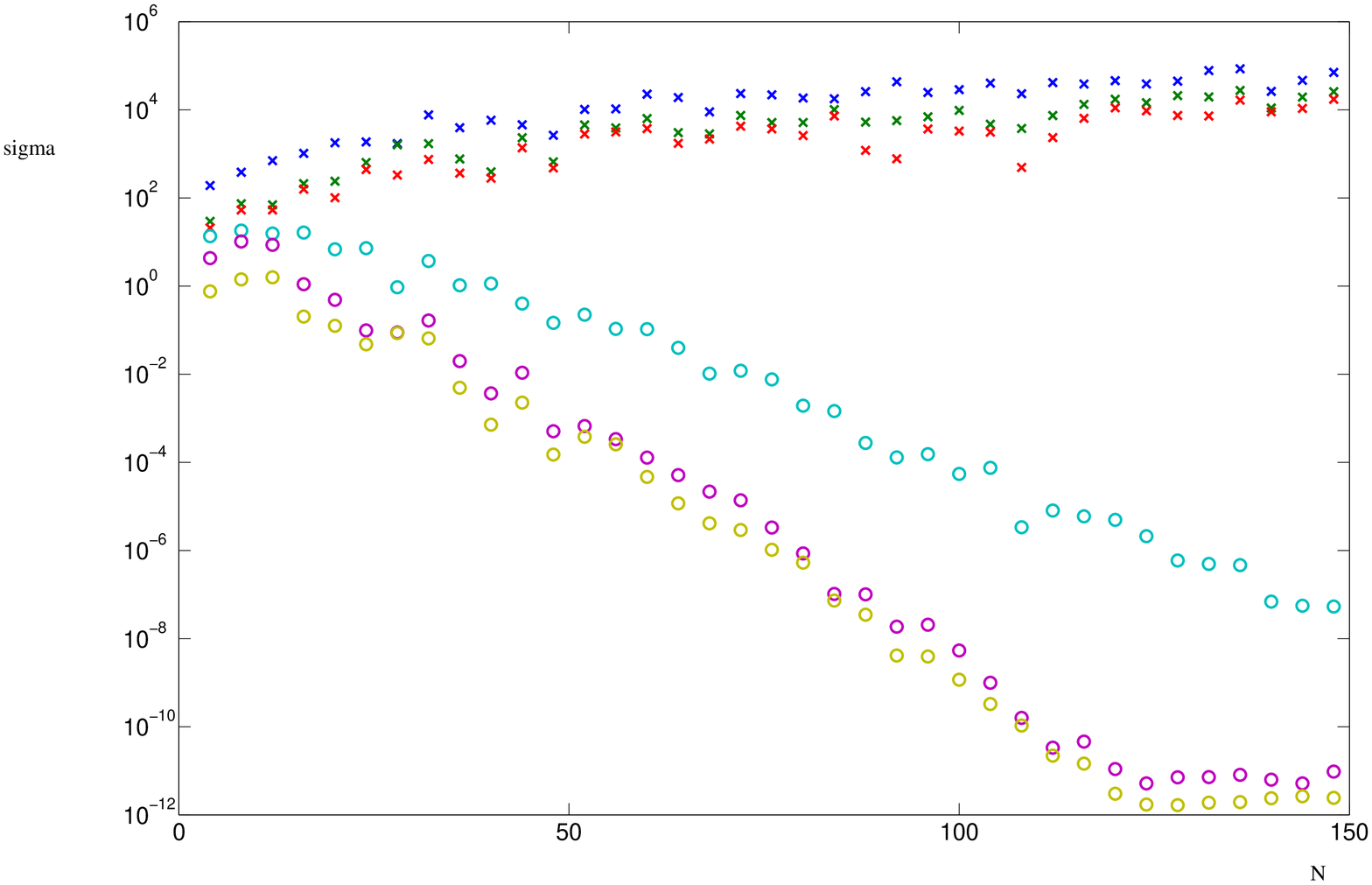}
\caption{\label{fig5} Example \ref{example5}. Eigenvalues from polyeig (open circles) and 
eigenvalues from
the integral algorithm 2 ($K=2$, filled boxes) for a quadratic matrix 
polynomial with rank defect
(left), singular values of integral algorithm $2$ with $l=3$ columns versus 
the number $N$ of quadrature nodes for
the same example (right). }
\end{figure}

\end{example}
\begin{figure}[h] 
\centering
\psfrag{real}{$\re$}
	\psfrag{imag}{$\im$}
\psfrag{N}{$N$}
	\psfrag{sigma}{$\sigma_j$}
\includegraphics[width=0.48\textwidth]{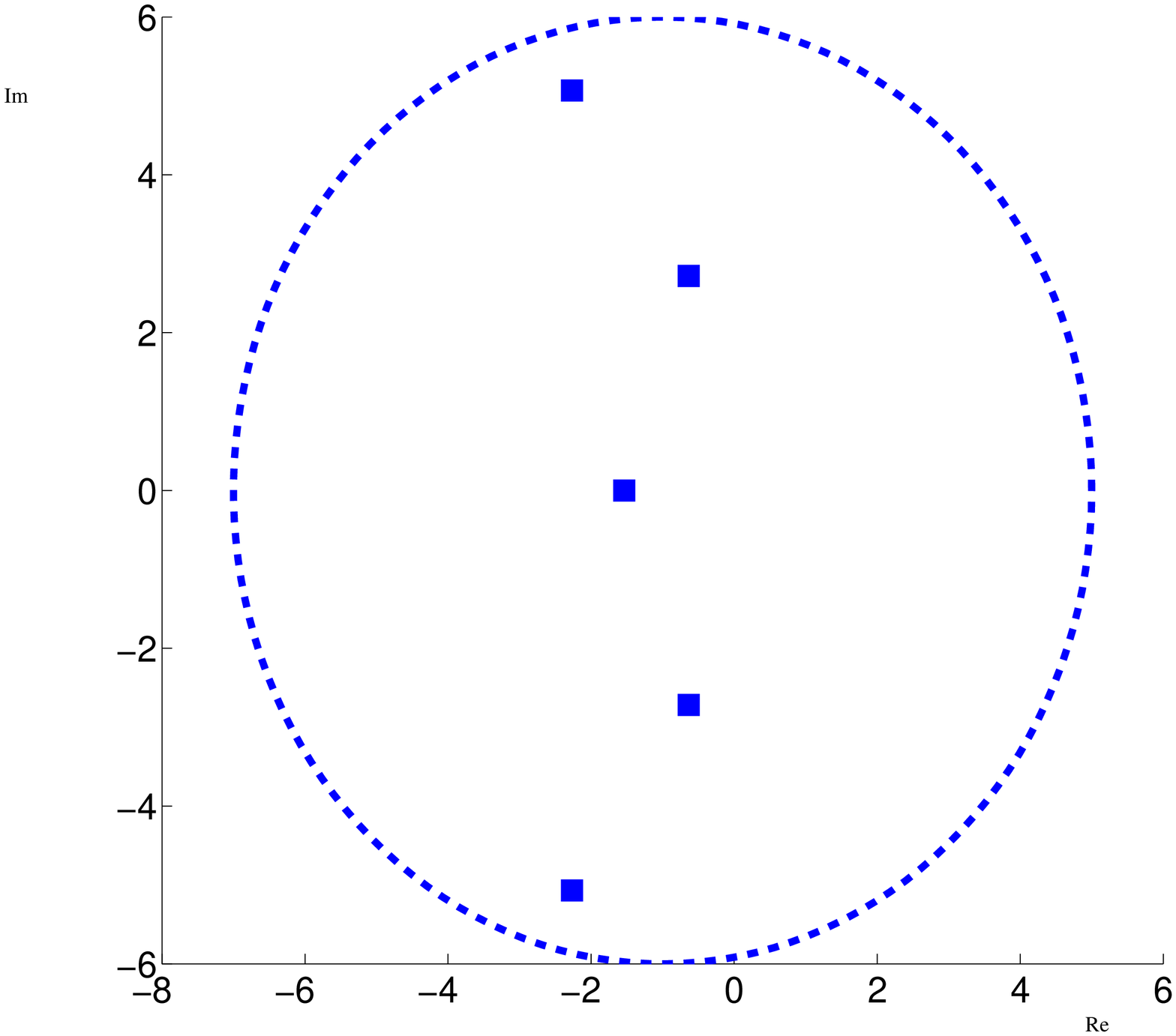} \quad
\includegraphics[width=0.48\textwidth]{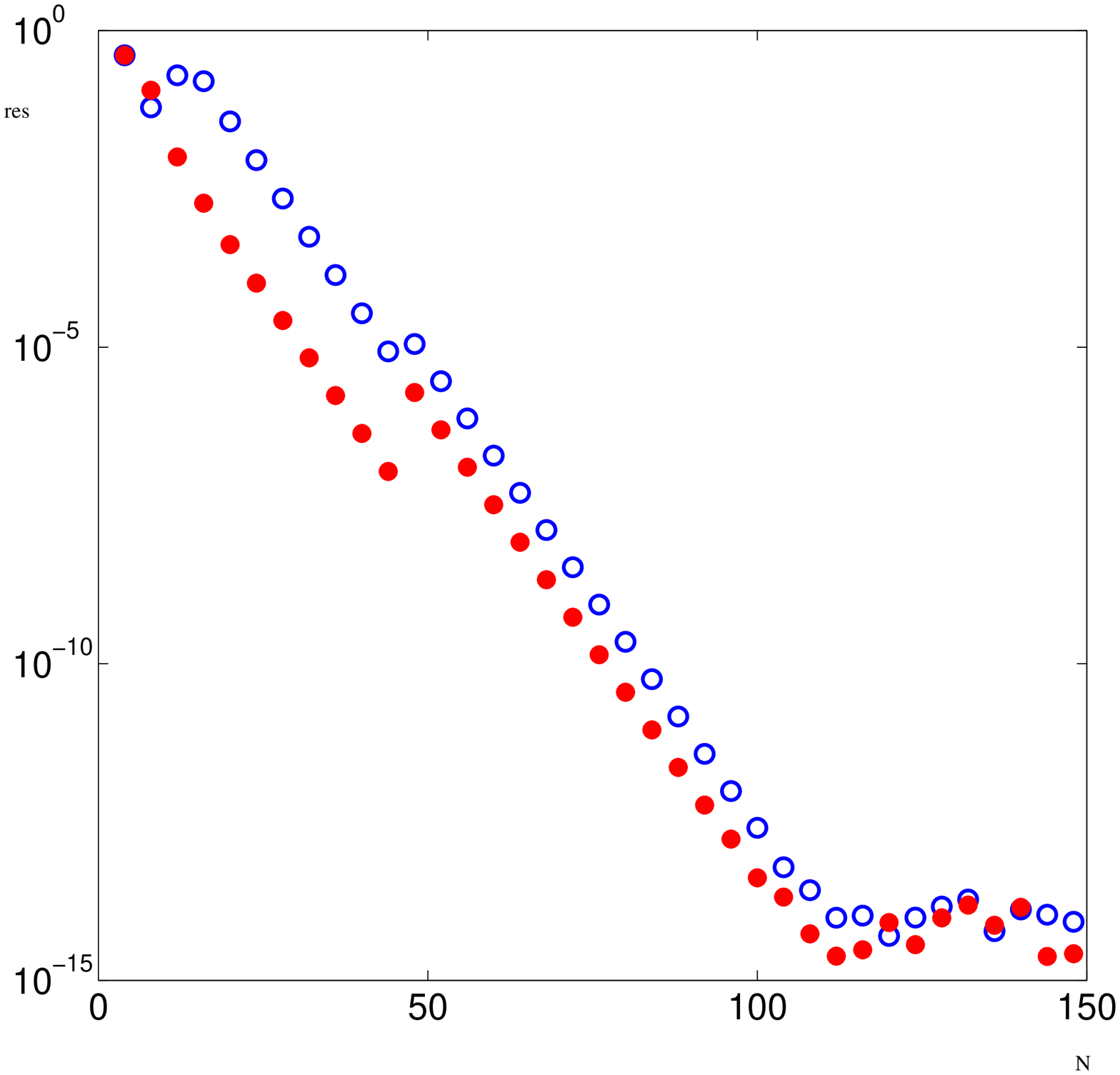}
\caption{\label{fig6} Example \ref{example6}.Eigenvalues of the characteric equation 
\eqref{delay} inside a circle of radius $6$ and with center $-1$, computed
with the integral algorithm 2 with $K=3,l=2$. 
(left), residuals $||T(\lambda_j)v_j||$ for $\lambda_1\approx -0.6+ 2.71 i$,
$\lambda_2 \approx -2.27 + 5.07 i$ versus   
the number $N$ of quadrature nodes for
the same example (right). }
\end{figure}

\begin{example} \label{example6}
Consider the characteristic equation of a delay system
$\dot{x}= T_0 x(t) +T_1 x(t-\tau)$ from \cite[Sec.2.4.2]{mn07},\cite{kr09},
given by
\begin{equation}\label{delay}
T(z)=z I - T_0 -T_1 e^{-z \tau}, \quad 
T_0=\begin{pmatrix} -5 & 1 \\ 2 & -6 \end{pmatrix},\quad
T_1=\begin{pmatrix} -2 & 1 \\ 4 & -1 \end{pmatrix}.
\end{equation}
\end{example}

In  case $\tau=1$ there are more than two eigenvalues
inside the circle $\varphi(t)= z_0 + R e^{it},\mu=-1,R=6$ .
We set $l=2,\hat{V}=I_2$ and $K=3$ for the integral algorithm 2 and
obtain with $N=150$ five eigenvalues inside the circle, 
(see Figure \ref{fig6}(left)),
which coincide with the computed ones in \cite{kr09}. Much smaller values
than $N=150$ give sufficient accuracy, since  there is a good separation
of singular values and a fast decay of residuals, see Figure \ref{fig6}(right).

\bibliographystyle{abbrv}
\bibliography{imev}

\def\cprime{$'$} \def\cprime{$'$} \def\cprime{$'$}
\begin{thebibliography}{10}

\bibitem{bhms08}
T.~Betcke, N.~J. Higham, V.~Mehrmann, C.~Schr{\"o}der, and F.~Tisseur.
\newblock {NLEVP}: {A} collection of nonlinear eigenvalue problems.
\newblock Technical Report 2008.40, MIMS, University of Manchester, Apr. 2008.
\newblock
  \texttt{www.mims.manchester.ac.uk/research/numerical-analysis/nlevp.html}.

\bibitem{bv04}
T.~Betcke and H.~Voss.
\newblock A {J}acobi-{D}avidson type projection method for nonlinear eigenvalue
  problems.
\newblock {\em Future Gener. Comput. Syst.}, 20:363--372, 2004.

\bibitem{D59}
P.~J. Davis.
\newblock On the numerical integration of periodic analytic functions.
\newblock In {\em On numerical approximation. {P}roceedings of a {S}ymposium,
  {M}adison, {A}pril 21-23, 1958}, Edited by R. E. Langer. Publication no. 1 of
  the Mathematics Research Center, U.S. Army, the University of Wisconsin,
  pages 45--59. The University of Wisconsin Press, Madison, 1959.

\bibitem{DR84}
P.~J. Davis and P.~Rabinowitz.
\newblock {\em Methods of numerical integration}.
\newblock Dover Publications Inc., Mineola, NY, 2007.
\newblock Corrected reprint of the second (1984) edition.

\bibitem{GS71}
I.~C. Gohberg and E.~I. Sigal.
\newblock An operator generalization of the logarithmic residue theorem and
  {R}ouch\'e's theorem.
\newblock {\em Mat. Sb. (N.S.)}, 84(126):607--629, 1971.

\bibitem{GK06}
R.~E. Greene and S.~G. Krantz.
\newblock {\em Function theory of one complex variable}, volume~40 of {\em
  Graduate Studies in Mathematics}.
\newblock American Mathematical Society, Providence, RI, third edition, 2006.

\bibitem{HHT08}
N.~Hale, N.~J. Higham, and L.~Trefethen.
\newblock Computing ${A}^{\alpha}$, $\log({A})$, and related matrix functions
  by contour integrals.
\newblock {\em SIAM J. Numer. Anal.}, 46:2505--2523, 2008.

\bibitem{h08}
N.~J. Higham.
\newblock {\em Functions of Matrices}.
\newblock SIAM, 2008.

\bibitem{Ke51}
M.~V. Keldysh.
\newblock On the characteristic values and characteristic functions of certain
  classes of non-self-adjoint equations.
\newblock {\em Doklady Akad. Nauk SSSR (N.S.)}, 77:11--14, 1951.

\bibitem{Ke71}
M.~V. Keldysh.
\newblock The completeness of eigenfunctions of certain classes of
  nonselfadjoint linear operators.
\newblock {\em Uspehi Mat. Nauk}, 26(4(160)):15--41, 1971.

\bibitem{kr09}
D.~Kressner.
\newblock A block {N}ewton method for nonlinear eigenvalue problems.
\newblock {\em Numer. Math.}, 114:355--372, 2009.

\bibitem{MS70}
A.~S. Markus and E.~I. Sigal.
\newblock The multiplicity of the characteristic number of an analytic operator
  function.
\newblock {\em Mat. Issled.}, 5(3(17)):129--147, 1970.

\bibitem{mv04}
V.~Mehrmann and H.~Voss.
\newblock Nonlinear eigenvalue problems: a challenge for modern eigenvalue
  methods.
\newblock {\em GAMM Mitteilungen}, 27, 2004.

\bibitem{MM84}
R.~Mennicken and M.~M{\"o}ller.
\newblock Root functions, eigenvectors, associated vectors and the inverse of a
  holomorphic operator function.
\newblock {\em Arch. Math. (Basel)}, 42(5):455--463, 1984.

\bibitem{MM03}
R.~Mennicken and M.~M{\"o}ller.
\newblock {\em Non-self-adjoint boundary eigenvalue problems}, volume 192 of
  {\em North-Holland Mathematics Studies}.
\newblock North-Holland Publishing Co., Amsterdam, 2003.

\bibitem{mn07}
W.~Michiels and S.-I. Niculescu.
\newblock {\em Stability and Stabilization of Time-delay Systems}, volume~12 of
  {\em Advances in Design and Control}.
\newblock Society for Industrial and Applied Mathematics (SIAM), 2007.

\bibitem{so06}
S.~Solovev.
\newblock Preconditioned iterative methods for a class of nonlinear eigenvalue
  problems.
\newblock {\em Linear Algebra Appl.}, 415:210--229, 2006.

\bibitem{stsu90}
G.~Stewart and J.~G. Sun.
\newblock {\em Matrix Perturbation Theory}.
\newblock Academic Press Inc., Boston, MA, 1990.

\bibitem{Tr68}
V.~P. Trofimov.
\newblock The root subspaces of operators that depend analytically on a
  parameter.
\newblock {\em Mat. Issled.}, 3(vyp. 3 (9)):117--125, 1968.

\bibitem{vo03}
H.~Voss.
\newblock A maxmin principle for nonlinear eigenvalue problems with application
  to a rational spectral problem in fluid-solid vibration.
\newblock {\em Appl. Math}, 48:607--622, 2003.

\bibitem{v04}
H.~Voss.
\newblock An {A}rnoldi method for nonlinear eigenvalue problems.
\newblock {\em BIT}, 44:387--401, 2004.

\bibitem{v07}
H.~Voss.
\newblock A {J}acobi-{D}avidson method for nonlinear and nonsymmetric
  eigenvalue problems.
\newblock {\em Comput. Struct.}, 85:1284--1292, 2007.

\bibitem{vw82}
H.~Voss and B.~Werner.
\newblock A minmax principle for nonlinear eigenvalue problems with
  applications to nonoverdamped systems.
\newblock {\em Math. Methods Appl. Sci.}, 4:415--424, 1982.

\end{thebibliography}
\end{document}